\documentclass[12pt,reqno,a4paper]{amsart}
\usepackage[utf8]{inputenc}
\usepackage[T1]{fontenc}
\usepackage{amssymb}
\usepackage{pgf,tikz,pgfplots}
\pgfplotsset{compat=1.12}
\usepackage{mathrsfs}
\usetikzlibrary{arrows}
\usepackage{bm}
\usepackage[alpine]{ifsym}
\usepackage{xpicture}
\usepackage{calculator}
\usepackage{graphicx}
\usepackage{enumerate}
\usepackage{enumitem}
\usepackage{acronym}
\usepackage{xspace}
\usepackage{xfrac}    
\usepackage{faktor} 
\usepackage{caption}
\usepackage{subcaption}

\usepackage[a4paper,top=3.3cm,bottom=3cm,left=3cm,right=3cm,bindingoffset=5mm]{geometry}

\usepackage{float}
\usepackage{color}

\usepackage[colorlinks=true]{hyperref}

\usepackage[normalem]{ulem}

\newtheorem{theorem}{Theorem}[section]
\newtheorem{lemma}[theorem]{Lemma}
\newtheorem{corollary}[theorem]{Corollary}
\newtheorem{proposition}[theorem]{Proposition}

\theoremstyle{definition}
\newtheorem{example}[theorem]{Example}
\newtheorem{remark}[theorem]{Remark}
\newtheorem{definition}[theorem]{Definition}
\newtheorem{algorithm}[theorem]{Algorithm}

\numberwithin{equation}{section}

\newcommand{\Z}{\mathbb{Z}}
\newcommand{\Q}{\mathbb{Q}}
\newcommand{\R}{\mathbb{R}}
\newcommand{\C}{\mathbb{C}}
\newcommand{\PP}{\mathbb{P}}

\def\fd{\mathbb{F}_\delta}

\def\fol{$\mathcal{F}$\xspace}
\def\folc{$\mathcal{F}^{\C^2}$\xspace}
\def\folp{$\mathcal{F}^{\PP^2}$\xspace}
\def\fold{$\mathcal{F}^{\delta}$\xspace}
\def\pf{\mathcal{P}_\mathcal{F}}

\def\df{D_\mathcal{F}}

\def\bf{\mathcal{B}_\mathcal{F}}

\def\zf{Z_\mathcal{F}}
\def\tf{T_{\mathcal{F}}}

\def\ipr{rational first integral\xspace}

\newcommand{\sn}{{\sum_{i=1}^n}}

\newcommand{\foldelta}{{\mathcal{F}^{\delta}}}
\newcommand{{\estrictfoldelta}}{{\widetilde{\mathcal{F}}}^{\delta}}

\title[Algebraic integrability with bounded genus]{Algebraic integrability with bounded genus}
\author[C. Galindo]{Carlos Galindo}
\address{Universitat Jaume I, Campus de Riu Sec, Departamento de Matem\'aticas \& Institut Universitari de Matem\`atiques i Aplicacions de Castell\'o, 12071
Caste\-ll\'on de la Plana, Spain.}\email{galindo@uji.es}   \email{callejo@uji.es}

\author[F. Monserrat]{Francisco Monserrat}
\address{Universitat Politècnica de València, Departament de Matemàtica Aplicada \& Institut Universitari de Matemàtica Pura i Aplicada,
Camí de Vera s/n, 46022 València (Spain).}
\email{framonde@mat.upv.es}

\author[E. P\'erez-Callejo]{Elvira P\'erez-Callejo}

\thanks{Partially supported by grants PID2022-138906NB-C22, TED2021-130358B-I00 and PRE2019-089907 funded by MCIN/AEI/10.13039/501100011033 and by “European Union NextGenerationEU/PRTR”, as well as by grants GACUJIMB/2023/03 and UJI-B2021-02 funded by Universitat Jaume I}

\begin{document}
\begin{abstract}
We provide an algorithm which decides whether a polynomial foliation \folc on the complex plane has a polynomial first integral of genus $g\neq 1$. Except in a specific case, an extension of the algorithm also decides if \folc has a rational first integral of that genus.
\end{abstract}

\maketitle

\section{Introduction}

At the end of the 19th century Darboux \cite{dar}, Poincaré \cite{poi11,poi12,poi13,poi21}, Painlevé \cite{pai} and Autonne \cite{aut} addressed the problem of solving complex planar polynomial differential systems. A natural question they asked was to characterize those systems such that every orbit of the attached vector field (or differential $1$-form) is algebraic. Equivalently, they wanted to decide whether a vector field as above has a rational first integral or, in other words, whether it is algebraically integrable. Of course, once confirmed the existence of a rational first integral, they also desired to compute it. 

Algebraic integrability presents the advantage that when a differential equation has a (rational) first integral, the study of its dynamics is reduced by a dimension. In addition, it is related to the second part of the XVI Hilbert problem \cite{lli2,lli}, to the Einstein's field equations in general relativity \cite{hew} and to the center problem for quadratic vector fields \cite{sch,cha-lli,lli2,lli}.

Nowadays the algebraic integrability problem for complex planar polynomial vector fields remains unsolved. To try to solve it, Poincaré proposed to give an upper bound on the degree of a general invariant curve and, in the strongest version, it is desirable that this bound depends only on the degree of the vector field. A complex planar vector field defines a foliation on the (affine) complex plane which can be extended to the complex projective plane. To solve the algebraic integrability problem for this last class of foliations is equivalent to solve it for the affine case; many of the recent  developments on algebraic integrability have been obtained in this context. The bound corresponding to the above mentioned strong version does not exist in general, even though the analytic types and the quantity of singularities of the foliation are fixed (see \cite{l-n}). 
However, an upper bound is known when the reduction of the singularities of the foliation on the complex projective plane has only one dicritical exceptional divisor \cite{GalMon2014}. Other papers dealing with algebraic integrability or the problem classically known as \emph{Poincaré problem} (which consists of giving an upper bound on the degree of the irreducible algebraic invariant curves in terms of the degree of the foliation and/or other data, without necessarily assuming algebraic integrability) are \cite{car, ca-ca,c-l,ce-li,es-kl,GalMon2010,pere,PerSva,soa2,Walcher,zam1,zam2}.

One of the natural ingredients to be added in the algebraic integrability problem is the (geometric) genus $g$ of the (normalization of) a general invariant algebraic curve. The genus of a non-singular projective variety $X$ is defined as $g_X=h^0(X,K_X)=\dim H^0(X,K_X)$, $K_X$ being the canonical sheaf of $X$. It is a birational invariant and extremely important for the classification problem \cite[Chapter II, Remark 8.18.3]{Har}.

Painlevé posed the problem of recognizing this value $g$ for an algebraically integrable foliation on the projective plane and, again, Lins-Neto in \cite{l-n} gave a negative answer by giving families of algebraically integrable foliations with analytic type fixed and genus arbitrarily large.

Returning to the idea of bounding the degree $d$ of a general invariant curve but also considering the genus, in \cite{pere} is proposed a bound on $d$ for foliations \folp of general type on the complex projective plane depending on $r$ (the degree of the foliation), $g$ and the sequence $\{h^0(\PP^2,K^{\otimes m}_{\mathcal{F}^{\PP^2}})\}_{m>0}$, where $K_{\mathcal{F}^{\PP^2}}$ is the canonical sheaf of \folp. This result was refined in \cite{PerSva}, where it is provided a enormous bound on $d$ for foliations \folp birationally equivalent to non-isotrivial fibrations of genus $\geq 2$. The bound only depends on $r$ and $g$.

The term \emph{effective algebraic integration} has been coined in the last decades and it groups the methods/algorithms to decide whether a planar vector field is algebraically integrable.

\emph{The \textbf{main results} in this paper are Algorithm \ref{algoritmo3} and Corollary \ref{rr2}. Algorithm \ref{algoritmo3} is an effective algebraic integration algorithm for a polynomial foliation \folc on the complex plane. It considers  an extended foliation \fol of \folc that one choose either on the projective plane $\PP^2$ or on any Hirzebruch surface $\fd$. Its input is the foliation \fol, a part of the reduction of the singularities of \fol (the dicritical configuration) and the genus $g\neq 1$, and the output (if a computable condition is not met), a rational first integral of genus $g$ if it exists. Corollary \ref{rr2} (together with Remark \ref{rr1}) shows that when one desires to know whether the foliation \folc has a rational first integral with fixed irreducible factors in the denominator or, even more interesting, whether \folc has a polynomial first integral (always of genus $g\neq 1$), the algorithm does not depend on any condition.}

Our algorithm is an effective algebraic integration algorithm for both foliations on the projective plane and on any Hirzebruch surface.

Taking a short tour on the algorithms on effective algebraic integration, many of them operate with bounded degree (of a general algebraic invariant curve) \cite{pr-si,man2,dua,d-l-a,fergia,cheze,bost}.

Invariant algebraic curves are important objects in most of the previous algorithms. An interesting fact is the result by Jouanolou \cite{Joua} that improves a result by Darboux \cite{dar} and states that, if our vector field has at least $\binom{d+1}{2}+2$ invariant algebraic curves, it is algebraically integrable and, from those curves, one can compute a rational first integral. In \cite{GalMon2006} it is given an algorithm that decides about algebraic integrability of a foliation on the projective plane (and computes a rational first integral in the affirmative case) when the cone of curves of the surface obtained by blowing-up at the dicritical configuration of the foliation is polyhedral. This algorithm allows us to compute a certain number of invariant algebraic curves which leads us to find the rational first integral.

As mentioned, we are mainly interested in complex planar polynomial differential systems which can be studied from the foliation \folc on the complex plane defined by the vector field attached to the differential system. As in the case of the complex projective plane, \folc also admits an extension to a foliation \fold on any Hirzebruch surface $\fd$ and Algorithm \ref{algoritmo} shows how to get \fold from \fol. Foliations \fold have an important role in this paper. 

Section \ref{secPrelim} contains information we will use later. Among other things, we briefly recall the concepts of foliation (Subsection \ref{secFol}) and Hirzebruch surface (Subsection \ref{secHirzSurf}) as well as an important input in our main algorithm, the dicritical configuration of a foliation, which is a part of the configuration involved in the reduction of its singularities \cite{Seiden}.

Let $S_0$ denote either the complex projective plane $\PP^2$ or a Hirzebruch surface $\fd$ and \fol a foliation on $S_0$, which can be regarded as the extension to $S_0$ of a foliation \folc as before. Section \ref{sec_ipr} is the core of the paper where we state the results which explain why Algorithm \ref{algoritmo3} works in the problem of effective integration with genus different from one. The key object here, given in (\ref{divisor}), is the divisor $\df$ on the surface $Z_{\mathcal{F}}$ obtained after blowing-up at the dicritical configuration $\mathcal{B}_{\mathcal F}$ of \fol (see (\ref{www})). When \fol is algebraically integrable, one can find a rational first integral from the complete linear system $|\df|$, which has projective dimension one. Subsection \ref{rfi} explains this fact and some interesting properties of $\df$ which are useful for our purposes, like $\df^2=0$ or $\df\cdot C=0$ for any {\fol-}invariant curve.

Then, what one desires is to find a candidate to be $\df$ in order to check out algebraic integrability. Subsection \ref{subsec_alg_int}, apart from Algorithm \ref{algoritmo3} and Corollary \ref{rr2}, also contains Algorithm \ref{algoritmo2} which decides about algebraic integrability on any foliation \fol under certain assumptions on a concrete $\Q$-divisor $T_{\alpha_{\mathcal F}^\Sigma}$ on $Z_{\mathcal F}$. Both algorithms have, as input, a restricted set $\Sigma$ of independent algebraic solutions (Definition \ref{def_restricted}), which is a suitable finite (even empty) set of integral invariant curves.

Our reasoning looks for candidates to be the $\Q$-divisor $T_{\mathcal F}=\frac{1}{b}\df$ introduced in (\ref{eqn_def_tf}), $b$ being a specific positive integer. When the foliation admits a rational first integral, the class $[T_{\mathcal F}]$ in the Néron-Severi space of $Z_{\mathcal F}$ is orthogonal to the classes in the set $V(\Sigma)$ introduced before Definition \ref{indep}. $V(\Sigma)$ contains the classes of the elements in $\Sigma$ and the strict transforms of the non-dicritical exceptional divisors. 

Assume now that we consider any non-necessarily algebraically integrable foliation \fol. The inputs in Algorithm \ref{algoritmo3} (from which we deduce Corollary \ref{rr2}) allow us to obtain a family $\{T_\alpha\}$ of candidates to be $T_{\mathcal F}$, where $\alpha$ is a parametrized vector on $\ell$ parameters, $\ell$ depending on the number of dicritical exceptional divisors and the cardinality of $\Sigma$. In addition $T_\alpha^2$ must vanish since $T_{\mathcal{F}}^2=0$. Then, with the help of Theorem \ref{teo1}, we prove in Lemma \ref{lemm} that the map $\alpha\rightarrow T_\alpha^2$ has an absolute maximum, only reached at $\alpha_{\mathcal{F}}^\Sigma$, value we considered before, which will give the candidate $T_{\alpha_{\mathcal F}^\Sigma}$ we use in our algorithms. Our main theorem, Theorem \ref{teo2}, states that if $T^2_{\alpha_{\mathcal F}^\Sigma}<0$, \fol is not algebraically integrable and if it is algebraically integrable and $T^2_{\alpha_{\mathcal F}^\Sigma}=0$, then $T_{\mathcal F}=T_{\alpha_{\mathcal F}^\Sigma}$.

Section \ref{sec_ipr} ends with our algorithm in bounded genus $g\neq 1$, Algorithm \ref{algoritmo3}, where we use the previous considerations and the case $T^2_{\alpha_{\mathcal F}^\Sigma}>0$ is also contemplated. We always get an answer except when certain inequality $p_{\inf}\cdot p_{\sup}\leq 0$ holds, in this case the algorithm gives no conclusive answer. However, Corollary \ref{rr2} shows that a minor modification always stops if one looks for a polynomial first integral of \folc. This is a particular case of that contemplated in Remark \ref{rr1}, where it is checked algebraic integrability (with bounded genus) for rational first integrals of the type $f/(f_1^{a_1}\cdots f_r^{a_r})$, where $f_1,\ldots,f_r$ are fixed polynomials and $a_i\in\Z_{>0}$.

It is worthwhile to add that our algorithms to decide about algebraic integrability on the complex plane are quite versatile. This is because one can choose to consider extensions to the projective plane or any Hirzebruch surface and, in each case, one can pick different restricted sets of independent algebraic solutions depending on the invariant curves one knows.

Section \ref{sec_examples} concludes the paper by giving several examples that show how our algorithms work. Examples \ref{ex_noip} and \ref{ex_ip1} give answer by using Algorithm \ref{algoritmo2}, Example \ref{ex_ip2} uses Algorithm \ref{algoritmo3} and Example \ref{ex_ip3} corresponds to Corollary \ref{rr2}. We have not been able to find any foliation satisfying the condition under which our algorithm gives no conclusive answer.

\section{Preliminaries} \label{secPrelim}

In this paper we study complex planar polynomial differential systems through their associated vector fields or  differential $1$-forms. Specifically, we are interested in algebraic integrability of polynomial foliations on the complex plane. We address this problem by extending our foliations to foliations on the projective plane or on Hirzebruch surfaces. In this section we recall some definitions and facts we will need. We start with the concept of foliation on a surface and the reduction procedure of its singularities. The last part of the section deals with Hirzebruch surfaces, foliations on these surfaces (or on the projective plane) and their algebraic integrability.

\subsection{Foliations}\label{secFol}

We recall one of the possible definitions for \emph{singular holomorphic foliation} (\emph{foliation}, in the sequel) and some related concepts and results.

Let $S$ be a smooth complex surface. A (singular holomorphic) \textit{foliation} $\mathcal{F}$ on $S$ is defined by a family of pairs $\{(U_i,v_i)\}_{i\in I}$, where $\{U_i\}_{i\in I}$ is an open covering of $S$  and, for all $i\in I$, $v_i$ is a non-vanishing holomorphic vector field on $U_i$ such that for any $i,j\in I$:
\begin{gather*}
v_i=g_{ij}v_j \text{ on }U_i\cap U_j\text{ for some element }g_{ij}\in\mathcal{O}_{S}(U_i\cap U_j)^*,
\end{gather*}
$\mathcal{O}_S$ being the sheaf of holomorphic functions on $S$.

A foliation is not exactly the collection $\{(U_i,v_i)\}_{i\in I}$, but an equivalence class: we identify two foliations given by collections $\{(U_i,v_i)\}_{i\in I}$ and $\{(U'_i,v'_i)\}_{i\in I}$ as above whenever $v_i$ and $v'_i$ coincide on $U_i\cap U'_i$ up to multiplication by a nowhere vanishing holomorphic function.

The functions $g_{ij}\in \mathcal{O}_{S}(U_i\cap U_j)^*$ define a commutative cocycle and, therefore, a class in $H^1(S,{\mathcal O}^*_S)$, that is, a line bundle on $S$. The sheaf associated to this line bundle is called the \emph{canonical sheaf} of $\mathcal F$ and will be denoted by ${\mathcal K}_{\mathcal F}$.

Given a point $p\in U_i$, the \textit{algebraic multiplicity} of $\mathcal{F}$ at $p$,  $\nu_p(\mathcal{F})$, is the order of the vector field $v_i$ at $p$. It is $t$ if and only if $(v_i)_p\in \mathcal{M}_p^t \Theta_{S,p}$ and $(v_i)_p\notin \mathcal{M}_p^{t+1} \Theta_{S,p}$, $\mathcal{M}_p$ (respectively, $\Theta_{S,p}$) being the maximal ideal of the local ring $\mathcal{O}_{S,p}$ (respectively, the tangent sheaf of $S$ at $p$). The \emph{singularities} of $\mathcal{F}$ are the points $p$ where $\nu_p(\mathcal{F})\geq 1$. In this paper any considered foliation has isolated singularities.

A foliation $\mathcal{F}$ on $S$ can also be defined by using $1$-forms. Indeed, it is given by a family $\{(U_i,\omega_i)\}_{i\in I}$ where $\{U_i\}_{i\in I}$ is, as above, an open covering of $S$ and, for all $i\in I$, $\omega_i$ is a non-zero regular differential $1$-form on $U_i$ such that, for any $i,j\in I$:
\begin{gather*}
\omega_i=f_{ij}\omega_j \text{ on }U_i\cap U_j\text{, for some element }f_{ij}\in\mathcal{O}_{S}(U_i\cap U_j)^*.
\end{gather*}

Let $C$ be a curve on $S$ and $\mathcal F$ a foliation on $S$ defined by a family $\{(U_i,v_i)\}_{i\in I}$ (respectively,$\{(U_i,\omega_i)\}_{i\in I}$ ) as above. Assume that $f_i=0$ is the equation of $C$ on $U_i$ for all $i\in I$. $C$ is said to be \emph{invariant} by \fol if, for all $i\in I$, it holds that $v_i(f_i)=h_if_i$ (respectively, $w_i\wedge df_i=f_i \eta_i$) for some regular function $h_i$ (respectively, differential $2$-form $\eta_i$) on $U_i$. We also say that $C$ is {\fol-}invariant.

\subsection{Configurations of infinitely near points}\label{subsec_conf}

Let $p$ be a point on a smooth complex surface $S$ and consider the blowup of $S$ centered at $p$, $\pi_p:{\rm Bl}_p(S)\rightarrow S$. The points in the exceptional divisor $E_{p}=\pi_p^{-1}(p)$ are named points in the \emph{first infinitesimal neighbourhood} of $p$. Assuming that we blow up at some points in $E_p$, the set of points in the new created exceptional divisors is called the \emph{second infinitesimal neighbourhood} of $p$. Inductively one defines the \emph{$k$th infinitesimal neighbourhood} of $p$, $k\geq 1$. A point $q$ is \emph{infinitely near} $p$ (denoted $q\geq p$) if either $q=p$ or it belongs to some $k$th infinitesimal neighbourhood of $p$. Finally, a point $q$ is \emph{infinitely near} $S$ if it is infinitely near a point in $S$. In addition, given two infinitely near $S$ points $p$ and $q$, $q$ \emph{is proximate to} $p$ (denoted $q\rightarrow p$) if $q$ belongs to $E_p$ or to any of its strict transforms. When $q$ is proximate to $p$ and $q\notin E_p$, $q$ is named \emph{satellite} and, otherwise, it is called \emph{free}. 

Consider a finite sequence
\begin{equation}\label{seq}
X_{n+1}\overset{\pi_n}{\rightarrow}X_{n}\overset{\pi_{n-1}}{\rightarrow}\cdots\overset{\pi_2}{\rightarrow}X_2\overset{\pi_1}{\rightarrow}X_1=S
\end{equation}
where, for all $i\in \{1,\ldots,n\}$, $p_i\in X_i$ and $\pi_i:X_{i+1}:={\rm Bl}_{p_i}(X_i)\rightarrow X_i$ is the blowup of $X_i$ centered at $p_i$. The set ${\mathcal C}=\{p_1,\ldots,p_n\}$ given by the centers of a finite sequence as (\ref{seq}) is called a \emph{configuration} (of infinitely near $S$ points). 

The \emph{proximity graph} of $\mathcal C$, denoted $\Gamma_{\mathcal C}$, is a labelled graph that provides a visual display of the proximity relations among the points in $\mathcal C$. It is constructed from bottom to top. The vertices at the bottom level, level $0$, correspond to the points of ${\mathcal C}\cap S$. Level $1$ corresponds to the points in $\mathcal C$ belonging to the first infinitesimal neighbourhoods of the points in ${\mathcal C}\cap S$; similarly, for $k\geq 2$, the vertices in the $k$th level correspond to the elements of $\mathcal C$ belonging to the $k$th infinitesimal neighbourhood of the points in ${\mathcal C}\cap S$. A vertex $p_i$ in a certain level is connected by an edge to a vertex $p_j$ in a higher level whenever $p_j$ is proximate to $p_i$. Each vertex is labelled $p_i$ according to the point in ${\mathcal C}$ it represents.

A configuration ${\mathcal C}$ is a \emph{chain} if $\geq$ defines a total order in $\mathcal C$, that is, there is only one vertex in each level of the proximity graph $\Gamma_{\mathcal C}$.

Finally, for any point $p_\ell\in{\mathcal C}$ the \emph{complete chain} associated to $p_\ell$ is the set
$$({\mathcal C})^{p_\ell}:=\{p_i\in {\mathcal C}\mid p_\ell \geq  p_i\}.$$

\subsection{Reduction of singularities}\label{subsec_red_sing}

Let $\mathcal{F}$ be a foliation on a smooth complex surface $S$. Let $p$ be a singularity of \fol. Assume that $\mathcal{F}$ is given at $p$, in local coordinates $x$ and $y$, by a local differential $1$-form $\omega =A(x,y) dx + B(x,y) dy=\sum_{k=m}^\infty \omega_k$, where  $\omega_k$ is the $k$th homogeneous jet of $\omega$. Then , the first non-vanishing jet of $\omega$ is $\omega_m :=a_m(x,y) dx + b_m(x,y) dy$, where $m=\nu_p(\mathcal{F})$. The point $p$ is a {\it simple singularity} of $\mathcal{F}$ whenever $m=1$ and the matrix
\[
 \left(
    \begin{array}{cc}
      \frac{\partial b_1}{\partial x} & \frac{\partial b_1}{\partial y} \\
     - \frac{\partial a_1}{\partial x} & - \frac{\partial a_1}{\partial y}\\
    \end{array}
  \right)
\]
has  two eigenvalues $\lambda_1, \lambda_2$ such that either the product $\lambda_1 \lambda_2$ does not vanish and their quotient is not a positive rational number, or $\lambda_1 \lambda_2 = 0$ and $\lambda_1^2 + \lambda_2^2$ does not vanish. Non-simple singularities are called {\it ordinary}. If the polynomial $d(x,y):= x a_m(x,y)+y b_m(x,y)$ is equal to zero, then we say that $p$ is a {\it terminal dicritical singularity}.  Notice that any terminal dicritical singularity is also an ordinary singularity.

The ordinary singularities of a foliation $ \mathcal{F}$ (recall that we assume that they are isolated) can be reduced by blowing-up. The following result summarizes well-known facts on reduction of singularities and terminal dicritical singularities (see \cite{Seiden,Bru} and \cite[Theorem 1 and Proposition 1]{fergal1}).

\begin{theorem}
\label{jmaa}
Let $\mathcal{F}$ be a foliation on a smooth complex surface $S$. Then:
\begin{enumerate}
\item There is a sequence of finitely many point blowups
\begin{equation}
Z=X_{n+1}\overset{\pi_n}{\rightarrow}X_{n}\overset{\pi_{n-1}}{\rightarrow}\cdots\overset{\pi_2}{\rightarrow}X_2\overset{\pi_1}{\rightarrow}X_1=S,
\end{equation}
such that the strict transform of $\mathcal{F}$  on $Z$ has no ordinary singularity.
\item A singularity $p$ of $\mathcal{F}$ is not terminal dicritical if and only if the exceptional divisor of the surface {\rm Bl}$_p(S)$ is invariant by the strict transform of $\mathcal{F}$ on {\rm Bl}$_p(S)$.
\end{enumerate}

\end{theorem}
A \emph{reduction of singularities} of $\mathcal{F}$ is a minimal (with respect to the number of involved blowups) map  $\pi: Z \rightarrow S$ as in Part ($1$) of Theorem \ref{jmaa}.\medskip

Let $\mathcal{S}_{\mathcal F}=\{p_1,\ldots,p_n\}$ be the configuration of centers of the reduction of singularities of $\mathcal{F}$. $\mathcal{S}_{\mathcal F}$ is called the \emph{singular configuration} of $\mathcal{F}$, and its elements \emph{infinitely near ordinary singularities} of $\mathcal{F}$.

We say that a point $p_i\in \mathcal{S}_{\mathcal{F}}$ is an \emph{infinitely near terminal dicritical singularity} if it is a terminal dicritical singularity of the strict tranform of $\mathcal{F}$ on the surface $p_i$ belongs to. Let  $\mathcal{D}_{\mathcal{F}}$ be the set of infinitely near terminal dicritical singularities and, for each $q\in \mathcal{D}_{\mathcal{F}}$, let us consider the complete chain
$$(\mathcal{S}_{\mathcal{F}})^q=\{p_i\in \mathcal{S}_{\mathcal{F}}\mid \mbox{$q\geq p_i$}\}.$$

The configuration $\mathcal{B}_{\mathcal{F}}:=\cup_{q\in \mathcal{D}_{\mathcal{F}}} (\mathcal{S}_{\mathcal{F}})^q$ is called \emph{dicritical configuration} of $\mathcal{F}$ and, its elements, \emph{infinitely near dicritical singularities}. Finally, the composition of the sequence of blowups centered at the points of $\mathcal{B}_{\mathcal F}$ will be named \emph{dicritical reduction of singularities} of $\mathcal{F}$.

\subsection{Hirzebruch surfaces}\label{secHirzSurf}

Let $\mathbb{F}_\delta:=\mathbb{P}(\mathcal{O}_{\mathbb{P}^1}\oplus \mathcal{O}_{\mathbb{P}^1}(\delta))$ be the $\delta$th complex Hirzebruch surface, $\delta\geq 0$, where $\mathbb{P}^1$ denotes  the complex projective line. There is a projection morphism $\sigma: \mathbb{F}_{\delta}\rightarrow \mathbb{P}^1$ giving $\fd$ a structure of ruled surface. $\mathbb{F}_{\delta}$ also has a structure of toric variety since it can be regarded as the quotient of $\left(\mathbb{C}^2\setminus \{\mathbf{0}\}\right)\times \left(\mathbb{C}^2\setminus \{\mathbf{0}\}\right)$ by an action of the algebraic torus $\left(\mathbb{C}\setminus \{0\}\right)\times \left(\mathbb{C}\setminus \{0\}\right)$. Considering coordinates $(X_0,X_1;Y_0,Y_1)$ in $\left(\mathbb{C}^2\setminus \{\mathbf{0}\}\right)\times \left(\mathbb{C}^2\setminus \{\mathbf{0}\}\right)$ and $(\lambda, \mu)\in\left(\mathbb{C}\setminus \{0\}\right)\times \left(\mathbb{C}\setminus \{0\}\right)$, the mentioned action is given by
\[
(\lambda,\mu)\cdot(X_0,X_1;Y_0,Y_1):=(\lambda X_0,\lambda X_1; \mu Y_0, \lambda^{-\delta} \mu Y_1).
\]

The Cox ring of $\fd$ (see \cite{Cox}) is the $\Z^2$-graded polynomial ring $\mathbb{C}[X_0,X_1,Y_0,Y_1]$, where $\deg X_0: = \deg X_1 :=(1,0)$, $\deg (Y_0):= (0,1)$  and $\deg Y_1: = (-\delta, 1)$. A polynomial in $\mathbb{C}[X_0,X_1,Y_0,Y_1]$ is \textit{bihomogeneous of bidegree} $(d_1,d_2)\in\Z^2$ if it belongs to the homogeneous piece corresponding to  $(d_1,d_2)$, that is, if it is a complex linear combination of monomials $X_0^{a_1}X_1^{a_2}Y_0^{b_1}Y_1^{b_2}$ with $a_1+a_2-\delta b_2=d_1$ and $b_1+b_2=d_2$.

The Picard group ${\rm Pic}(\fd)$ of $\fd$ (identified with the group of divisors modulo numerical equivalence) is generated by the classes of two divisors, $F$ and $M$, where $F$ is a fiber of $\sigma$ and $M$ a section such that $M^2=\delta$; thus $F^2=0$ and $F\cdot M=1$. Throughout the paper, ${\mathcal O}_{\mathbb{F}_{\delta}}(d_1,d_2)$ denote the invertible sheaf ${\mathcal O}_{\fd}(d_1F+d_2M)$.

The effective classes in ${\rm Pic}(\fd)$ are  $d_1[F]+d_2[M]$, where  $(d_1,d_2)\in \mathbb{Z}^2$ satisfies $d_1 + \delta d_2 \geq 0$ and $d_2 \geq 0$ and, for any divisor $D$, $[D]$ stands for the class of $D$ in the Picard group \cite{Har}. Given an effective class $d_1[F]+d_2[M]$, the set of non-zero elements in the space of global sections of ${\mathcal O}_{\fd}(d_1,d_2)$ is naturally identified with the set of bihomogeneous polynomials of bidegree $(d_1,d_2)$.

A Hirzebruch surface is covered by four affine open sets $U_{ij},\;i,j\in\{0,1\}$:
\begin{equation}\label{eqn_open_cover}
U_{ij}:=\{(X_0,X_1;Y_0,Y_1)\in\fd \;|\;X_i\neq 0 \text{ and }Y_j\neq 0\},
\end{equation}
where we identify $(X_0,X_1;Y_0,Y_1)$ with its image by the natural quotient map $\left(\mathbb{C}^2\setminus \{\mathbf{0}\}\right)\times \left(\mathbb{C}^2\setminus \{\mathbf{0}\}\right)\rightarrow \fd$ associated with the above mentioned action. Notice that, in $U_{00}$, \[
(X_0,X_1;Y_0, Y_1)=(1,X_1/X_0;1,X_0^\delta Y_1/Y_0)\]
and thus, we can identify the points on $U_{00}$ with the points $(x,y)$ on $\C^2$ through the isomorphism $(1,x;1,y)\rightarrow (x,y)$. A similar identification holds for any other open set $U_{ij}$ giving rise to change of coordinates maps in the overlap of two open sets $U_{ij}$ (see, for instance, \cite{GalMonOliv} for more details).

\subsection{Foliations on Hirzebruch surfaces}

Let \fold be a foliation on a Hirzebruch surface $\mathbb{F}_{\delta}$. By \cite{GalMonOliv}, \fold is determined by a differential $1$-form
\[
\Omega^\delta=A_{\delta,0} dX_0 + A_{\delta,1}  dX_1 + B_{\delta,0} dY_0 + B_{\delta,1} dY_1,
\]
where $A_{\delta,0}$, $A_{\delta,1}$, $B_{\delta,0}$ and $B_{\delta,1}$ are bihomogeneous polynomials in $\C[X_0,X_1,Y_0,Y_1]$ without non-constant common factors (not all of them equal to 0) of respective bidegrees $(d_1 - \delta + 1,d_2 + 2)$, $(d_1 - \delta + 1,d_2 + 2)$, $(d_1 - \delta + 2, d_2 + 1)$ and $(d_1+2, d_2 + 1)$, with $d_1,d_2\in \mathbb{Z}$. Moreover they must satisfy the following two conditions:
\begin{gather*}
A_{\delta,0} X_0+A_{\delta,1} X_1 - \delta B_{\delta,1} Y_1 =0 \mbox{ and } B_{\delta,0} Y_0 + B_{\delta,1} Y_1=0.
\end{gather*}
Notice that ${\mathcal K}_{\mathcal F}={\mathcal O}_{\mathbb{F}_{\delta}}(d_1,d_2)$.

Finally, if $x=\frac{X_1}{X_0}$ and $y=\frac{X_0^\delta Y_1}{Y_0}$ are the affine coordinates in $U_{00}\cong \mathbb{C}^2$ considered in Subsection \ref{secHirzSurf}, then the foliation on $U_{00}$ defined by the restriction of \fold is given by the differential $1$-form $A_{\delta,1}(1,x,1,y)dx+B_{\delta,1}(1,x,1,y)dy$. There are analogous expressions for the restriction of \fold to the remaining affine open subsets $U_{ij}$, see \cite{GalMonOliv} for details.   

\begin{remark}
Foliations on Hirzebruch surfaces behave in a similar way to those on the complex projective plane $\PP^2$. Let $(X:Y:Z)$ be projective coordinates, then $\PP^2$ can be covered by three affine open sets $U_{W}=\{(X:Y:Z)\in\PP^2\,|\,W\neq 0\}$, where $W$ equals $X,\,Y$ or $Z$. We can identify each one of these affine open sets with $\C^2$. Then, it is well-known that a foliation \folp (of degree $r$) on $\PP^2$ is determined by a differential $1$-form \[
\Omega^{\PP^2}:=A_{\PP^2}dX+B_{\PP^2}dY+C_{\PP^2}dZ,\]
where $A_{\PP^2},\,B_{\PP^2}$ and $C_{\PP^2}$ are homogeneous polynomials (not all zero) of the same degree $r+1$ in the polynomial ring $\C[X,Y,Z]$ without non-constant common factors and satisfying the following condition (called Euler's condition):\[
A_{\PP^2}X+B_{\PP^2}Y+C_{\PP^2}Z=0.\]
\end{remark}

\subsection{Polynomial foliations on $\mathbb{C}^2$ and their extensions to Hirzebruch surfaces}

Consider a polynomial differential system on $\mathbb{C}^2$
\begin{equation}\label{poly}
\dot{x}=P(x,y),\;\; \dot{y}=Q(x,y),
\end{equation}
where $P,Q$ are polynomials in $\mathbb{C}[x,y]$ without non-constant common factors or, equivalently, the planar vector field $P(x,y)\frac{\partial}{\partial x}+Q(x,y)\frac{\partial}{\partial y}$ or the differential $1$-form $Q(x,y)dx-P(x,y)dy$. The system (\ref{poly}) defines a foliation on the affine plane $\mathbb{C}^2$ and we call this type of foliations, \emph{polynomial foliations} on $\mathbb{C}^2$.

Let \folc be a polynomial foliation on $\mathbb{C}^2$ and consider a Hirzebruch surface $\mathbb{F}_{\delta}$. Identifying the affine plane with the affine open subset $U_{00}\subseteq \mathbb{F}_{\delta}$ described in Subsection \ref{secHirzSurf}, \folc can be regarded as a foliation on $U_{00}$ and it admits an extension to $\mathbb{F}_{\delta}$. We denote this extension by ${\mathcal F}^\delta$ and it can be computed by means of the following algorithm, \cite[Section 3]{GalMonPerC2022}.

\begin{algorithm}\label{algoritmo}

$ $\vspace{0.25cm}

\textbf{Input:} A non-negative integer $\delta$, a differential $1$-form $A(x,y)dx+ B(x,y)dy$ (where $A(x,y), B(x,y)$ are coprime polynomials in $\mathbb{C}[x,y]$) defining a polynomial foliation \folc on $\C^2$.

\textbf{Output:} Polynomials $A_{\delta,0}, A_{\delta,1}, B_{\delta,0}, B_{\delta,1} \in\C[X_0,X_1,Y_0,Y_1]$ such that the $1$-form $\Omega^\delta=A_{\delta,0}dX_0+ A_{\delta,1}dX_1+ B_{\delta,0}dY_0+ B_{\delta,1}dY_1$ defines a foliation ${\mathcal F}^\delta$ on $\mathbb{F}_{\delta}$ whose restriction to $U_{00}\cong \mathbb{C}^2$ coincides with \folc. \\

\begin{enumerate} [label=(\arabic*)]

\item Write the rational functions $A\left(\frac{X_1}{X_0}, \frac{X_0^{\delta} Y_1}{Y_0} \right)$ and $B\left(\frac{X_1}{X_0}, \frac{X_0^{\delta} Y_1}{Y_0} \right)$ as reduced rational fractions $\frac{X_0^{\alpha_0}A_{\delta, 1}}{X_0^{\alpha_1} Y_0^{\alpha_2}}$ and $\frac{X_0^{\beta_0}B_{\delta, 1}}{X_0^{\beta_1} Y_0^{\beta_2}}$, respectively, where $(\alpha_0,\alpha_1,\alpha_2)$ and $(\beta_0,\beta_1,\beta_2)$ belong to $\mathbb{Z}_{\geq 0}^3$ and $A_{\delta, 1}$ and $B_{\delta, 1}$ are bigraded homogeneous polynomials in $\C[X_0,X_1,Y_0,Y_1]$ of respective bidegrees $(\lambda_1,\lambda_2)=(\alpha_1-\alpha_0,\alpha_2)$ and $(\mu_1,\mu_2)=(\beta_1-\beta_0,\beta_2)$ such that $A_{\delta,1},\;B_{\delta,1}$ and $X_0 Y_0$ are pairwise coprime.

\item Let $m_1:=\lambda_1-\mu_1+1+\delta$. If $m_1>0$, then $B_{\delta, 1}:=X_0^{m_1}B_{\delta, 1}$; otherwise, $A_{\delta, 1}:=X_0^{-m_1}A_{\delta, 1}$.

\item Let $m_2:=\lambda_2-\mu_2-1$. If $m_2>0$, then $B_{\delta, 1}:=Y_0^{m_2}B_{\delta, 1}$; otherwise, $A_{\delta, 1}:=Y_0^{-m_2}A_{\delta, 1}$.

\item Let $\gamma_2:=0$ if $Y_0$ divides $B_{\delta, 1}$, and $\gamma_2:=1$ otherwise. Set $B_{\delta, 1}:= Y_0^{\gamma_2} B_{\delta, 1}$ and $A_{\delta, 1}:=Y_0^{\gamma_2}A_{\delta, 1}$.

\item Let $\gamma_1:=0$ if $X_0$ divides $\delta Y_1 B_{\delta, 1}-X_1 A_{\delta, 1}$ and $\gamma_1:=1$ otherwise. Set $A_{\delta, 1}:=X_0^{\gamma_1} A_{\delta, 1}$ and $B_{\delta, 1}:=X_0^{\gamma_1} B_{\delta, 1}$.

\item Set $A_{\delta, 0}:=\frac{\delta Y_1 B_{\delta, 1}-X_1 A_{\delta, 1}}{X_0}$ and $B_{\delta, 0}:=\frac{-Y_1 B_{\delta, 1}}{Y_0}$.

\item Return $A_{\delta, 0}, A_{\delta, 1}, B_{\delta, 0}$ and $B_{\delta, 1}$.
\end{enumerate}
\end{algorithm}

\begin{remark}
If the differential $1$-form given in the input of the preceding algorithm is $dx$ (respectively, $dy$), then the polynomial $B_{\delta,1}=0$ (respectively, $A_{\delta,1}=0$) must be considered, by convention, as a global section of the sheaf ${\mathcal O}_{\mathbb{F}_{\delta}}(-\delta-1,-1)$, i.e., $(\mu_1,\mu_2)=(\delta-1,-1)$ (respectively $(\lambda_1,\lambda_2)=(\delta-1,-1)$) at Step (1); this is because, otherwise, the obtained $1$-form $\Omega^{\delta}$ would not be reduced. 

\end{remark}

\subsection{Rational first integrals}

 We start this subsection by defining the concept of \emph{rational first integral} of a foliation (respectively, polynomial foliation) on $\PP^2$ or $\fd$ (respectively, on $\mathbb{C}^2$). We also recall some of their properties.

\begin{definition}
Let $S$ denote either the (complex) plane $\mathbb{C}^2$, the projective plane $\PP^2$ or a Hirzebruch surface $\mathbb{F}_{\delta}$, and let $\mathcal{F}$ be a foliation on $S$ (which is assumed to be polynomial if $S=\mathbb{C}^2$). A non-constant rational function $h$ of $S$ is a \emph{rational first integral} of $\mathcal{F}$ if its indeterminacy set is contained in $\mathcal{S}_{\mathcal F}\cap S$ and its level curves contain the leaves of $\mathcal F$. If $\mathcal{F}$ admits a rational first integral we say that $\mathcal{F}$ is \emph{algebraically integrable}.
\end{definition}

Let \folc be an algebraically integrable polynomial foliation on $\mathbb{C}^2$ defined by a polynomial vector field $X:=P(x,y)\frac{\partial}{\partial x}+Q(x,y)\frac{\partial}{\partial y}$ (or, equivalently, by the associated $1$-form $\omega=Q(x,y)dx-P(x,y)dy$). A rational first integral of  \folc is, then, given by a non-constant (reduced) rational function $R=\frac{F}{G}\in \mathbb{C}(x,y)$ such that $X(R)=0$ (or, equivalently, $\omega \wedge dR=0$). We say that a rational function $R\in \mathbb{C}(x,y)$ is \emph{composite} if there exists a rational function (in one variable) $u=u_1(t)/u_2(t)\in \mathbb{C}(t)$, of degree $\deg(u):=\deg(u_1)-\deg(u_2)$ greater than 1, such that $R=u\circ H$, where $H\in \mathbb{C}(x,y)\setminus \mathbb{C}$. Otherwise $R$ is said to be \emph{non-composite}. If a foliation \folc is algebraically integrable, then the set of rational first integrals of \folc is  closed under composition with non-constant rational functions of $\mathbb{C}(t)$ and, moreover, a reduced rational function $R\in \mathbb{C}(x,y)$ is a rational first integral of \folc if and only if $u\circ R$ is a rational first integral of \folc for some non-constant rational function $u\in \mathbb{C}(t)$. In addition, there exists a non-composite rational first integral $R$ of \folc such that any other rational first integral has the form $u\circ R$ for some $u\in \mathbb{C}(t)\setminus \mathbb{C}$; then we say that $R$ is a \emph{primitive} rational first integral of \folc. Proofs of these facts can be seen in \cite{bost}.

The closures of the level curves of a primitive rational first integral $R=\frac{F}{G}$ of \folc, namely those defined by the equations $\alpha F(x,y)+\beta G(x,y)=0$ (where $(\alpha:\beta)$ runs over $\mathbb{P}^1$), form a \emph{pencil} of curves in $\mathbb{C}^2$, which we denote by ${\mathcal P}_{\mathcal{F}^{\C^2}}$. 

Regarding $\mathbb{C}^2$ as the affine open subset $U_Z$ (respectively, $U_{00}$) of $\PP^2$ (respectively, $\mathbb{F}_{\delta}$), the rational function $R$ corresponds to a rational first integral $R(X/Z,Y/Z)$ (respectively, $R(X_1/X_0,X_0^\delta Y_1/Y_0)$) of the extension \fol of \folc to $\PP^2$ (respectively, $\mathbb{F}_{\delta}$). As before, its level curves define a pencil of curves $\mathcal{P}_{\mathcal{F}}$ in $\PP^2$ (respectively, $\mathbb{F}_\delta$).

The reduced and irreducible invariant by \folc (respectively, \fol) curves are exactly the irreducible components of the curves in $\mathcal{P}_{\mathcal F^{\C^2}}$ (respectively, $\mathcal{P}_{\mathcal{F}}$). In addition, all the curves of the pencil $\mathcal{P}_{\mathcal F^{\C^2}}$ (respectively, $\mathcal{P}_{\mathcal{F}}$), except finitely many, are reduced and irreducible.

\section{Rational first integrals with given genus}\label{sec_ipr}

From now on, we denote by $S_0$ a surface which is either $\PP^2$ or $\fd$, $\delta\geq 0$. In this section we introduce the pieces for proving our main results. These are Algorithm \ref{algoritmo3} and Corollary \ref{rr2} and are also stated at the end of this section. Consider a polynomial foliation \folc on $\mathbb{C}^2$. \emph{Through the extension \fol of \folc to $S_0$, Algorithm \ref{algoritmo3} allows us to decide (except in a particular situation) whether \folc has a rational first integral of given genus $g\neq 1$ (and computes it in the affirmative case)}. Corollary  \ref{rr2} proves that a minor modification of Algorithm \ref{algoritmo3} decides about polynomial integrability without any exception. Recall that the \emph{genus of a rational first integral} of \folc is the geometric genus of (the compactification of) a general curve of the pencil $\mathcal{P}_{\mathcal F^{\C^2}}$ and it is a birational invariant. The existence of a \ipr of a foliation \folc is equivalent to the existence of a \ipr of its extended foliation \fol, and one can be computed from the other.

\subsection{The $\mathbb{Q}$-divisor $\tf$ associated to an algebraically integrable foliation \fol}\label{rfi}

In this subsection, \folc is a polynomial foliation on $\mathbb{C}^2$ which admits a (primitive) rational first integral $R$. We are going to give some information concerning the pencil ${\mathcal P}_{\mathcal F}$ (and, as a consequence, the first integral) coming from the global geometry of the surface obtained from the dicritical configuration of the extended foliation \fol. In particular we will define a $\Q$-divisor attached to \fol which will be very useful in our next subsection on algebraic integrability with bounded genus.

If $S_0=\PP^2$ (respectively, $\fd$), we consider the rational first integral of \fol given by $\frac{F}{G}=R(X/Z,Y/Z)$ (respectively, $R(X_1/X_0, X_0^\delta Y_1/Y_0)$),  $F, G$ being coprime polynomials in $\C[X,Y,Z]$ (respectively, $\mathbb{C}[X_0,X_1,Y_0,Y_1]$). Then, the pencil ${\mathcal P}_{\mathcal{F}}$ is generated by the curves on $S_0$ with equations $F=0$ and $G=0$ and the assignment $P\mapsto (F(P):G(P))$ gives rise to a rational map $\phi: S_0\dashrightarrow \mathbb{P}^1$ whose indeterminacy locus is supported at the set of base points of ${\mathcal P}_{{\mathcal F}}$ (which is finite). By \cite[Theorem II.7]{Beau}, there exists a finite sequence of point blowups 
\begin{equation}\label{www}
Z_{\mathcal{F}}=Z_{n+1}\rightarrow Z_{n}\rightarrow\cdots\rightarrow Z_2\rightarrow Z_1:=S_0,
\end{equation}
and a morphism $\varphi: Z_{{\mathcal F}}\rightarrow \mathbb{P}^1$ that eliminates the indeterminacies of $\phi$, that is, if $\pi_{\mathcal{F}}$ denotes the composition of the blowups in (\ref{www}), it holds that $\varphi=\phi\circ \pi_{\mathcal{F}}$. Equivalently, $\pi_{\mathcal{F}}$ eliminates the base points of the pencil ${\mathcal P}_{\mathcal{F}}$ (see the proof of \cite[Theorem II.7]{Beau}, or \cite[Sections 5 and 6]{fergal1} for a very explicit description of this process). 

It turns out that the configuration given by the centers of the blowups in (\ref{www}) coincides with the dicritical configuration $\mathcal{B}_{\mathcal{F}}$ associated to \fol. This fact is deduced from \cite{TesJulio} and, also, it is proved in \cite[Corollary 2]{fergal1} (although the latter result is shown for foliations on the projective plane, the proof is local in nature and, therefore, it can be trivially adapted for foliations on a Hirzebruch surface). Hence, \emph{the dicritical reduction of singularities of $\mathcal{F}$ provides the surface $Z_{\mathcal{F}}$ corresponding to the elimination of base points of the pencil $\mathcal{P}_{\mathcal{F}}$}. Let us write
$$\mathcal{B}_{\mathcal{F}}=\{p_1,\ldots,p_n\}.$$

Throughout the paper, given a curve $C$ on $S_0$ or on any of the intermediate surfaces $Z_i$, $2\leq i\leq n$, we will denote by $C^*$ (respectively, $\widetilde{C}$) the total (respectively, strict) transform of $C$ on the surface $Z_{\mathcal{F}}$. Also, we denote by $E_i$ the exceptional divisor produced by the blowup $Z_{i+1}\rightarrow Z_i$ of the sequence (\ref{www}). Denote by $L$ a line on $\PP^2$ and keep the above notation. Then, the Picard group of $Z_{\mathcal{F}}$ is generated by $[L^*], [E_1^*], \ldots,[E^*_n]$ (respectively, $[F^*], [M^*], [E^*_1],\ldots,[E^*_n]$) whenever $S_0=\PP^2$ (respectively, $S_0=\fd$), where $[\cdot]$ stands for numerical equivalence class. Recall that, for all $i,r\in \{1,\ldots,n\}$, ${E}^*_i\cdot {E}^*_r=-\delta_{ir}$, $\delta_{ir}$ being the Kronecker's delta, and $\widetilde{E}_i=E_i^*-\sum_{p_r\rightarrow p_i} E_r^*$.

Consider the divisor, also on $Z_{\mathcal F}$,
\begin{equation}\label{divisor}
D_{\mathcal{F}}:=bL^*-\sum_{i=1}^n m_i E_i^*\,\,\left(\text{respectively, }aF^*+b M^*-\sum_{i=1}^n m_i E_i^*\right)
\end{equation}
if $S_0=\PP^2$ (respectively, $S_0=\fd$) such that the above morphism $\varphi$ coincides with the morphism $\phi_{|D_{\mathcal{F}}|}$ induced by the complete linear system $|D_{\mathcal{F}}|$ (up to the choice of a suitable basis). Then, the strict transform on $Z_{\mathcal{F}}$ of a general curve of the pencil ${\mathcal P}_{\mathcal{F}}$ is linearly equivalent to $D_{\mathcal{F}}$ (see the proof of Theorem II.7 in \cite{Beau}). Hence, $\pf\subseteq |bL|$ (respectively, ${\mathcal P}_{\mathcal{F}}\subseteq |aF+bM|$) whenever $S_0=\PP^2$ (respectively, $S_0=\fd$), and all curves in ${\mathcal P}_{\mathcal{F}}$ but finitely many (which we call \emph{special curves}) are reduced and irreducible. Moreover, the multiplicity of their strict transforms at $p_i$ is $m_i$, $1\leq i\leq n$.

Notice that $\df$ is a nef divisor on $Z_{\mathcal F}$ and, as a consequence, $b> 0$ (respectively, $a\geq 0$ and $b\geq 0$) if $S_0=\PP^2$ (respectively, $S_0=\fd$) \cite{Har}. We name $\df$ the \emph{characteristic divisor} of \fol.

An exceptional divisor $E_i$ is \emph{dicritical} (with respect to ${\mathcal F}$) if $p_i$ is an infinitely near terminal dicritical singularity; otherwise $E_i$ is non-dicritical.

The next two lemmas provide some properties of the divisor $\df$.
\begin{lemma}\label{clave}
Keep the above notation and assumptions. Let $\widetilde{\mathcal{F}}$ be the strict transform of \fol and $C$ a curve both on $Z_{\mathcal{F}}$. The following conditions are equivalent:
\begin{itemize}
    \item[(a)] $C$ is $\widetilde{\mathcal{F}}$-invariant.
    \item[(b)] The integral (reduced and irreducible) components of $C$ are either strict transforms of $\mathcal{F}$-invariant integral curves or strict transforms of non-dicritical (with respect to ${\mathcal F}$) exceptional divisors.
    \item[(c)] $D_{\mathcal{F}}\cdot C=0$.
\end{itemize}
\end{lemma}

\begin{proof}
To prove the equivalence between (a) and (b), it is enough to show that $\widetilde{E}_i$ is $\widetilde{\mathcal{F}}$-invariant if and only if $E_i$ is non-dicritical. This fact is proved in \cite[Proposition 1]{fergal1} for $S_0=\PP^2$ and the same arguments are valid for $S_0=\fd$.

Let us show the equivalence between (b) and (c). Without loss of generality we can assume that $C$ is an integral curve on $Z_{\mathcal{F}}$. 
On the one hand, if $C$ is the strict transform of a curve $C'$ on $S_0$, then one has that $C$ is $\widetilde{\mathcal{F}}$-invariant if and only if $C'$ is $\mathcal{F}$-invariant. This happens if and only if $C$ is a component of a fiber of the morphism $\varphi$, which is equivalent to say that $C$ is contracted by $\varphi$ (that is, $D_{\mathcal{F}}\cdot C=0$). On the other hand, applying \cite[Proposition 2]{fergal1} (whose proof is also true within our framework), if $C=\widetilde{E}_i$ for some $i\in \{1,\ldots,n\}$, it holds that $E_i$ is non-dicritical if and only if $m_i-\sum_{p_r\rightarrow p_i}m_r=0$, that is, if and only if $D_{\mathcal{F}}\cdot \widetilde{E}_i=0$.

\end{proof}

\begin{lemma}\label{clave2}
Keeping the above notation and assumptions, the following equalities hold:
\begin{itemize}
\item[(a)] $D_{\mathcal F}^2=0$.
\item[(b)] $(\pi_{\mathcal{F}})_*|D_{\mathcal{F}}|=\mathcal{P}_{\mathcal{F}}$.
\end{itemize}

\end{lemma}

\begin{proof}

Part (a) is because two general elements of the linear system $|\df|$ do not meet. 

To prove Part (b), notice that the inclusion $\mathcal{P}_{\mathcal{F}}\subseteq (\pi_{\mathcal{F}})_*|D_{\mathcal{F}}|$ holds because, for any curve $C\in \pf$, either $\widetilde{C}\in |\df|$ or there exists an effective divisor $E$ with exceptional support such that $\widetilde{C}+E\in |\df|$ (see the proof of Theorem II.7 in \cite{Beau}). 

To prove the equality we reason by contradiction. Then, assume that $(\pi_{{\mathcal F}})_*|D_{\mathcal{F}}|\setminus \pf$ is not empty (and, therefore, infinite). Let us show that the integral components of any curve $H\in (\pi_{{\mathcal F}})_*|\df|\setminus \pf$ are components of a special curve of the pencil ${\mathcal P}_{{\mathcal F}}$. Indeed, if $H_1$ is such a component then $D_{{\mathcal F}}\cdot \widetilde{H}_1=0$ (as a consequence of Part (a) and the fact that $D_{{\mathcal F}}$ is a nef divisor). By Lemma \ref{clave}, there exists a curve $F\in \pf$ such that $H_1$ is an integral component of $F$. Finally, notice that $F$ is a special curve of ${\mathcal P}_{{\mathcal F}}$ because, otherwise, $\widetilde{F}$ would be an element of the complete linear system $|\df|$ and, therefore, $F$ would be equal to $H$ (which is a contradiction because $H\not\in \pf$). Since there are finitely many special curves in $\pf$, there are finitely many possible curves $H$ as above, which is a contradiction with the fact that the set $(\pi_{{\mathcal F}})_*|D_{\mathcal{F}}|\setminus \pf$ is infinite.

\end{proof} 

The following result shows that for deciding about algebraic integrability of foliations on Hirzebruch surfaces, \emph{one can assume, without loss of generality, that their dicritical configurations are not empty}. By Bézout's Theorem, this assumption is also valid for foliations on $\PP^2$.

\begin{proposition}\label{befe}

Let \fold be a foliation on a Hirzebruch surface $\fd$. Assume that the dicritical configuration $\mathcal{B}_{\foldelta}$ of $\foldelta$ is empty. Then, either $\foldelta$ is the foliation defined by the fibers of the ruling $\mathbb{F}_{\delta}\rightarrow \mathbb{P}^1$ given by the projection $(X_0,X_1;Y_0,Y_1)\mapsto (X_0,X_1)$, or $\delta=0$ and $\mathcal F^0$ is defined by the fibers of the ruling $\mathbb{F}_{0}\cong\mathbb{P}^1\times \mathbb{P}^1 \rightarrow \mathbb{P}^1$ given by the projection $(X_0,X_1;Y_0,Y_1)\mapsto (Y_0,Y_1)$. 

Moreover, when $\mathcal{B}_{\foldelta}\neq \emptyset$, the coefficient $b$ in the expression given in (\ref{divisor}) is different from 0.

\end{proposition}

\begin{proof}
Firstly, assume that $\mathcal{B}_{\foldelta}$ is empty and, hence, $Z_{\foldelta}=\mathbb{F}_\delta$. Let $D_{\foldelta}=aF+bM$, with $a,b\in \mathbb{Z}$. Then, by Lemma \ref{clave2}, $0=D_{\foldelta}^2=2ab+b^2\delta$ and, therefore, either $b=0$ or $a=-b\delta/2$. 

In the first case, as the projective dimension of $(\pi_{{\mathcal F}^\delta})|D_{\foldelta}|={\mathcal P}_{\foldelta}$ is $1$, one has that $a=1$ and, then, ${\mathcal P}_{\foldelta}$ is the pencil of curves with equations $\alpha X_0+\beta X_1=0$, where $(\alpha:\beta)$ runs over $\mathbb{P}^1$. This means that the algebraic $\foldelta$-invariant curves are given by the fibers of the natural ruling $\mathbb{F}_\delta\rightarrow \mathbb{P}^1$.

In the second case (that is, $a=-b\delta/2$), $\delta$ must be zero because, otherwise, $D_\foldelta\cdot M_0<0$ (which is a contradiction because the linear system $|D_\foldelta|$ has no base point). Then $D_\foldelta=bM$ and a similar reasoning as in the above paragraph shows that $b=1$ and the algebraic  $\foldelta$-invariant curves are given by the fibers of the projection defined by $(X_0,X_1,Y_0,Y_1)\mapsto (Y_0,Y_1)$.

The last statement holds because, if ${\mathcal B}_{\foldelta}\neq \emptyset$ and $b=0$, then $D_{\foldelta}^2<0$, which is a contradiction by Lemma \ref{clave2}.

\end{proof}

We introduce more notation that will be used in the next subsection. Let \fol be an algebraically integrable foliation on $S_0$ and $\bf\neq \emptyset$ its dicritical configuration. Then, by Bézout's Theorem if $S_0=\PP^2$, or by Proposition \ref{befe} if $S_0=\fd$, $b$ is different from zero (with notation as in (\ref{divisor})). Denote by $\tf$ the $\mathbb{Q}$-divisor $\tf:=\frac{1}{b}\df$ on $Z_{\mathcal F}$, $\df$ being the characteristic divisor of \fol. Thus, $\tf$ equals
\begin{equation}\label{eqn_def_tf}
{\everymath={\displaystyle}
\begin{array}{c}
L^*-\sum_{i=1}^{n}s_i{E}_i^*,\,\text{ if }S_0=\PP^2\text{ and}\\
h F^*+M^*-\sum_{i=1}^{n}s_i{E}_i^*,\,\text{ if }S_0=\fd,
\end{array}}
\end{equation}
where $h:=a/b\in\Q$ and $s_i:=m_i/b\in \Q_{>0}$ for all $i$.

 Recall that $\bf=\{p_0,\ldots,p_n\}$. For each $r\in \{1,\ldots,n\}$, let $q_r\in {\mathcal B}_{\mathcal F}$ be the unique point in $S_0$ such that $p_r\geq q_r$. Let us consider the following divisor with exceptional support:
 \begin{equation*}
\hat{E}_r:=\sum_{i=1}^n a_{ir} E_i^*,
\end{equation*}
where, for all $i=1,\ldots,n$, $a_{ir}:={\rm mult}_{p_i}(\varphi_r)$, $\varphi_r$ denoting an analytically irreducible germ of curve in ${\mathcal O}_{S_0,q_r}$ whose strict transform is transversal to the divisor $E_r$ at a general point, and ${\rm mult}_{p_i}(\varphi_r)$ is the multiplicity of its strict transform at $p_i$. Notice that $a_{ir}=0$ if $p_i\not\in ({\mathcal B}_{\mathcal F})^{p_r}$, $a_{rr}=1$ and $a_{ir}=\sum_{p_\ell\rightarrow p_i} a_{\ell r}$ for all $i$ such that  $p_i\in ({\mathcal B}_{\mathcal F})^{p_r}\setminus \{p_r\}$. See Subsection \ref{subsec_conf} for the concepts of complete chain and proximity. The set $\hat{\mathcal E}:=\{\hat{E}_1,\ldots,\hat{E}_n\}$ is a basis of the free $\mathbb{Z}$-module $\bigoplus_{i=1}^n \mathbb{Z} E_i^*$ and  \cite[Lemma 8.4.5]{Cas} shows that 
 $$\sum_{i=1}^n m_i E_i^*=\sum_{j=1}^d \rho_{k_j} \hat{E}_{k_j},$$
 where $p_{k_1},\ldots,p_{k_d}$ denote the (infinitely near) terminal dicritical singularities of \fol and $\rho_{k_j}:=m_{k_j}-\sum_{p_\ell\rightarrow p_{k_j}} m_\ell$ for all $j=1,\ldots,d$. Notice that $\rho_{k_j}>0$ for all $j$ by Lemma \ref{clave} (because $\widetilde{E}_{k_j}$ is not {\fol-}invariant and $\df$ is a nef divisor). Hence, we can write the $\mathbb{Q}$-divisor $\tf$ in the following form:
\begin{equation*}
{\everymath={\displaystyle}
\begin{array}{c}
 \tf=L^*-\sum_{j=1}^d \beta_j \hat{E}_{k_j},\,\text{ if }S_0=\PP^2\text{ and}\\
 \tf=hF^*+M^*-\sum_{j=1}^d \beta_j \hat{E}_{k_j},\text{ if }S_0=\fd,
\end{array}}
\end{equation*}
 where $\beta_j:=\rho_{k_j}/b$. Thus, for all $j=1,\ldots,d$, we can rewrite 
 \begin{equation}\label{epsilonpico}
\hat{E}_{k_j}=\sum_{i=1}^n \lambda_{ij} E_i^*
 \end{equation}
 and then
\begin{equation}\label{tt2}
{\everymath={\displaystyle}
\begin{array}{c}
 \tf=L^*- \sum_{i=1}^n \left(\sum_{j=1}^d \lambda_{ij}\beta_j\right) E_i^*,\,\text{ if }S_0=\PP^2\text{ and}\\
 \tf=hF^*+M^*- \sum_{i=1}^n \left(\sum_{j=1}^d \lambda_{ij}\beta_j\right) E_i^*,\,\text{ if }S_0=\fd,
\end{array}}
\end{equation}
 where $\lambda_{ij}:=a_{i k_j}$ for all $i=1,\ldots,n$.
 
 \begin{remark}\label{lambdas}
 The basis $\hat{\mathcal E}$ and the values $\lambda_{ij}$ given in (\ref{epsilonpico}) are defined \emph{regardless of whether or not $\mathcal F$ is algebraically integrable}. Moreover, the values $\lambda_{ij}$ can be computed directly from the proximity graph of the configuration $\mathcal{B}_{\mathcal F}$.
\end{remark}

\subsection{An algorithm for algebraic integrability with bounded genus}\label{subsec_alg_int}

In this subsection, we state our main results on algebraic integrability announced at the beginning of the section. We also give the main ingredients for stating and proving our results. Let \folc be a polynomial foliation on $\C^2$ (which needs not to be algebraically integrable). Let $S_0$ be either the projective plane $\PP^2$ or a Hirzebruch surface $\fd$, $\delta\geq 0$, and consider \fol the extension of \folc to $S_0$. Let ${\mathcal K}_{\mathcal F}={\mathcal O}_{\PP^2}(r-1)$ (respectively, ${\mathcal O}_{\mathbb{F}_\delta}(d_1,d_2)$) if $S_0=\PP^2$ (respectively, $S_0=\fd$) be its canonical sheaf.

Assume, as mentioned above, that \emph{the dicritical configuration ${\mathcal B}_{\mathcal F}$ is not empty} (notice that Proposition \ref{befe} describes the rational first integrals of \fol whenever $S_0=\fd$ and ${\mathcal B}_{\mathcal F}=\emptyset$). Let $d$ be the number of \emph{(infinitely near) terminal dicritical singularities}.

Keep the notation in the previous subsections. In particular, suppose that
$${\mathcal B}_{\mathcal F}=\{p_1,\ldots,p_n\}$$
and denote by $$p_{k_1},\ldots,p_{k_d}\in {\mathcal B}_{\mathcal F}$$  the terminal dicritical singularities of \fol. 

Set $NS(Z_{\mathcal F})\cong{\rm Pic}(Z_{\mathcal F})\otimes_{\mathbb{Z}} \mathbb{R}$ the Néron-Severi space of the surface $Z_{\mathcal F}$. If $S_0=\PP^2$ (respectively, $\fd$), it is a real vector space of dimension $n+1$ (respectively, $n+2$) (the Picard number of $Z_{\mathcal F}$). We identify the class $[D]$ in ${\rm Pic}(Z_{\mathcal F})$ of each divisor $D$ on $Z_{\mathcal F}$ with its image in $NS(Z_{\mathcal F})$.

The canonical sheaf of $S_0$ is ${\mathcal O}_{\PP^2}(-3)$ (respectively, ${\mathcal O}_{\mathbb{F}_\delta}(\delta-2,-2)$) \cite{Har} and, therefore, the canonical sheaf of $Z_{\mathcal F}$ is ${\mathcal O}_{Z_{\mathcal F}}(K_{Z_{\mathcal F}})$, where $K_{Z_{\mathcal F}}$ equals $-3L^*+\sum_{i=1}^n E_i^*$ (respectively, $(\delta-2)F^*-2M^*+\sum_{i=1}^n E_i^*$) whenever $S_0=\PP^2$ (respectively, $\fd$). Finally, the canonical sheaf of the strict transform $\widetilde{\mathcal F}$ of \fol on $\zf$ is ${\mathcal K}_{\widetilde{\mathcal F}}:={\mathcal O}_{Z_{\mathcal F}}(K_{\widetilde{\mathcal F}})$, where $K_{\widetilde{\mathcal F}}$ is
\begin{gather*}
(r-1)L^*-\sum_{i=1}^n (\nu_{p_i}(\mathcal F)+\epsilon_{p_i}(\mathcal F)-1)E_i^*,\,\text{ when }S_0=\PP^2\text{, and}\\
d_1F^*+d_2M^*-\sum_{i=1}^n (\nu_{p_i}(\mathcal F)+\epsilon_{p_i}(\mathcal F)-1)E_i^*\,\text{ otherwise,}
\end{gather*}
$\nu_{p_i}(\mathcal F)$ being the multiplicity at $p_i$ of the strict transform of \fol on the surface containing $p_i$, and $\epsilon_{p_i}(\mathcal F)$ equals $1$ (respectively, $0$) if $p_i$ is a terminal dicritical singularity (respectively, otherwise) \cite[Proposition 1.1]{cam-car-gar}.

Given a finite set $\Sigma$ of integral curves on $Z_{\mathcal{F}}$, $V(\Sigma)$ denotes the following subset of $NS(Z_{\mathcal{F}})$:
$$V(\Sigma):=\{[C]\mid C\in \Sigma\}\cup \left\{[K_{\widetilde{\mathcal F}}-K_{Z_{\mathcal F}}]\right\}\cup\left\{[\widetilde{E}_i]| \mbox{$E_i$ is non-dicritical}
\right\}.$$

The following concept will be useful.

\begin{definition}\label{indep}
A \textit{set of independent algebraic solutions of} \fol (of length $\sigma$) is a set $\Sigma=\{C_1,\ldots,C_\sigma\}$ of integral curves on $S_0$, {\fol}-invariant, and such that the elements in the set $V(\Sigma)$ are $\mathbb{R}$-linearly independent. The empty set is also an independent set of algebraic solutions.
\end{definition}

\begin{remark}\label{K-K}
If \fol is algebraically integrable, then the divisor $K_{\widetilde{\mathcal F}}-K_{Z_{\mathcal F}}$ is linearly equivalent to a linear combination of $\widetilde{\mathcal F}$-invariant curves \cite[Lemma 1.1]{serrano} (see also \cite[page 3621]{GalMon2014}). 
\end{remark}

\begin{remark}
Notice that there are $n-d$ non-dicritical exceptional divisors and, therefore, when $S_0=\PP^2$ (respectively, $\fd$), the cardinality $\sigma$ of a set $\Sigma$ of independent algebraic solutions is, at most, $d$ (respectively, $d+1$). This is because the Picard number of $Z_{{\mathcal F}}$ is $n+1$ (respectively, $n+2$). However, if ${\mathcal F}$ is algebraically integrable, $\sigma\leq d-1$ (respectively, $\sigma\leq d$) because the codimension of the linear span of $V(\Sigma)$ is, at least, $1$; indeed, in this case, $V(\Sigma)$ is contained in a hyperplane of $NS(Z_{\mathcal F})\cong \R^{n+1}$ (respectively, $\mathbb{R}^{n+2}$) by Lemma \ref{clave}.
\end{remark}

We just saw that when $S_0=\PP^2$ (respectively, $\fd$), $\sigma>d-1$ (respectively, $\sigma>d$) implies that the foliation \fol is not algebraically integrable. Therefore, \emph{throughout the rest of the paper, we will assume that all considered sets of independent algebraic solutions have cardinality $\sigma$ less than or equal to $d-1$ (respectively, $d$) in case that $S_0=\PP^2$ (respectively, $\fd$)}. We will not mention it explicitly. The maximum possible value for $\sigma$ is denoted by $\sigma_{\max}$. \medskip

The following proposition shows that, choosing a suitable Hirzebruch surface $\mathbb{F}_{\delta}$, one can find non-empty sets of independent algebraic solutions. For any non-negative integer $\delta$, let us denote by $C_{X_0}$ (respectively, $C_{Y_0}$) the curve on $\mathbb{F}_{\delta}$ with equation $X_0=0$ (respectively, $Y_0=0$).

\begin{proposition} \label{prop_X0}
Assume that \folc is the polynomial foliation on $\C^2$ defined by the $1$-form $\omega=A(x,y)dx+B(x,y)dy$ and let \fold be its extension to $\mathbb{F}_{\delta}$. Then, the following statements hold:

\begin{enumerate}
\item For all $\delta\in \mathbb{Z}_{\geq 0}$, except for at most one value, $C_{X_0}$ is an {\fold}-invariant curve. 
\item Assume $B(x,y)\neq 0$, then $C_{Y_0}$ is an {\fold}-invariant curve  if and only if $\deg_yA(x,y)\leqslant \deg_yB(x,y)+1$. Here, $\deg_y$ means degree in the variable $y$ of a polynomial in $(\C[x])[y]$.
\end{enumerate}
\end{proposition}

\begin{proof}
Let $\Omega_\delta=A_{\delta,0}dX_0+A_{\delta,1}dX_1+B_{\delta,0}dY_0+B_{\delta,1}dY_1$ the $1$-form defining \fold, where $A_{\delta,0},\;A_{\delta,1},\;B_{\delta,0},\;B_{\delta,1}$ are the output of Algorithm \ref{algoritmo}. $C_{X_0}$ is {\fold}-invariant if and only if $X_0$ is a factor of $A_{\delta,1},\;B_{\delta,0}$ and $B_{\delta,1}$.

If $\gamma_1=1$ in Step $(5)$ of Algorithm \ref{algoritmo}, then $X_0$ divides $A_{\delta,1}$ and $B_{\delta,1}$, and, by Step $(6)$, also divides $B_{\delta,0}$. If $\gamma_1=0$ in Step $(5)$, it means that $X_0$ divides $\delta Y_1B_{\delta,1}-X_1A_{\delta,1}$. If $X_0$ does not divide $B_{\delta,1}$ neither $A_{\delta,1}$, it may divide $\delta Y_1B_{\delta,1}-X_1A_{\delta,1}$ for only one single value of $\delta$. This proves Part ($1$).

To show Part ($2$), notice that $C_{Y_0}$ is {\fold}-invariant if and only if $Y_0$ is a factor of $A_{\delta,0},\;A_{\delta,1}$ and $B_{\delta,1}$. If, in Step $(3)$, $m_2\leq 0$, then, at the beginning of Step $(4)$, $Y_0$ divides $A_{\delta,1}$ and does not divide $B_{\delta,1}$. It means that $\gamma_2=1$ in Step $(4)$ and then $Y_0=0$ is invariant. If $m_2>0$ in Step $(3)$, $Y_0$ does not divide $A_{\delta,1}$ because $\gamma_2=0$ in Step $(4)$. As $\deg_yA(x,y)- \deg_yB(x,y)-1=m_2$ in Step $(3)$, the proof becomes complete.

\end{proof}

Next we introduce some notation that will be used later. Given a subset $A$ of $NS(Z_{\mathcal F})$, we set \[
A^\perp:=\{{\bold x}\in NS(Z_{\mathcal F})\,|\,{\bold a}\cdot {\bold x}=0 \text{ for all }{\bold a}\in A\}.\]
 Also, given a divisor $D$ on $Z_{\mathcal F}$, $[D]_{=1}$ denotes the affine hyperplane of $NS(Z_{\mathcal F})$ formed by those ${\bold x}\in NS(Z_{\mathcal F})$ such that $[D]\cdot {\bold x}=1$.

When ${\mathcal F}$ is algebraically integrable on $\PP^2$ (respectively, $\fd$) and $\Sigma$ is a set of independent algebraic solutions of ${\mathcal F}$, it holds that the class $[T_{{\mathcal F}}]$ of the divisor $\tf$ introduced at the end of Subsection \ref{rfi} belongs to $V(\Sigma)^\perp\cap [G]_{=1}$, where $G=L^*$ (respectively, $F^*$). This property follows from the next Lemma \ref{clave3}. 

The following definition refines the concept of set of independent algebraic solutions of a foliation on $S_0$.

\begin{definition}\label{def_restricted}
Let \fol be the extension to $S_0=\PP^2$ (respectively, $\fd$) of a polynomial foliation on $\C^2$. A \emph{restricted set of independent algebraic solutions of \fol} (of length $\sigma$) is a set $\Sigma$ of $\sigma$ independent algebraic solutions of \fol such that $[L^*]$ (respectively, $[F^*]$) does not belong to the linear span of $V(\Sigma)$.
\end{definition}

\begin{lemma}\label{clave3}
Let $\mathcal{F}$ be the extension to $S_0=\PP^2$ (respectively, $\fd$) of a polynomial foliation on $\mathbb{C}^2$. Consider a set $\Sigma=\left\{C_1, \ldots, C_\sigma\right\}$ of independent algebraic solutions of $\mathcal{F}$ and set $\ell:=d-\sigma-1$ (respectively, $d-\sigma$), $d$ being the number of terminal dicritical singularities of $\mathcal{F}$. Then the following statements hold:
\begin{itemize}
\item [(a)]  $\operatorname{dim}_{\mathbb{R}} V(\Sigma)^{\perp}=\ell+1$.
\item[(b)] If $\Sigma$ is a restricted set of independent algebraic solutions, then $V(\Sigma)^{\perp} \cap\left[G\right]_{=1}$ is an affine subspace of $NS(\zf)$ of dimension $\ell$, where $G=L^*$ (respectively, $F^*$).
\item[(c)] If $\Sigma=\emptyset$, then $\Sigma$ is a restricted set of independent algebraic solutions of $\mathcal{F}$.
\item[(d)] If $\mathcal{F}$ is algebraically integrable, then $\Sigma$ is a restricted set of independent algebraic solutions of $\mathcal{F}$, and the class of the divisor $T_{\mathcal{F}}$ (introduced below Proposition \ref{befe}) satisfies
$$
\left[T_{\mathcal{F}}\right] \in V(\Sigma)^{\perp} \cap\left[G\right]_{=1} ,
$$
where $G=L^*$ (respectively, $F^*$).
\end{itemize}
\end{lemma}

\begin{proof}
We prove the lemma for the case $S_0=\fd$. The case when $S_0=\PP^2$ runs similarly by replacing $F$ by $L$ and considering the integer $\ell$ and the corresponding divisors $K_{\mathcal F},\,K_{Z_{\mathcal F}}$ and $\df$.

Recall that \[
V(\Sigma)=\left\{[C]|C\in\Sigma\right\}\cup\{[K_{\widetilde{F}}-K_{S_{\mathcal{F}}}]\}\cup\{[\widetilde{E}_i]|E_i\text{ is non-dicritical}\}\]
and hence, $\#V(\Sigma)=\sigma+1+n-d$. Since the set $V(\Sigma)$ is free, $\sigma+1+n-d$ is the rank of the matrix whose rows are the coordinates (in the basis $\{F^*,M^*\}\cup\{E_i^*\}_{1\leq i\leq n}$) of the vectors in $V(\Sigma)$. Then, considering the system of linear equations
\begin{equation}\label{conditions}
{\mathbf a}\cdot {\mathbf x}=0,\;{\mathbf a}\in V(\Sigma),
\end{equation}
one gets $\dim_{\mathbb{R}}\; V(\Sigma)^\perp=n+2-\#V(\Sigma)=\ell+1$ which proves (a).

Part (b) follows from the fact that, if $\Sigma$ is a restricted set of independent algebraic solutions, then the system of linear equations that results from adding the equation $[F^*]\cdot {\mathbf x}=1$ to the set of equations given in (\ref{conditions}) is consistent.

Let us prove (c). Denote by $\{E_{k_j}\}_{j=1}^{n-d}$ the set of non-dicritical divisors. By definition $V(\emptyset)=\{[K_{\widetilde{\mathcal F}}-K_{S_{\mathcal{F}}}]\}\cup\{[\widetilde{E}_{k_j}]\}_{j=1}^{n-d}$ and we know that 
$K_{\widetilde{\mathcal{F}}}-K_{S_{\mathcal{F}}}$ is linearly equivalent to $$(d_1-\delta+2)F^*+(d_2+2)M^*-\sum_{i=1}^n (\nu_{p_i}(\mathcal{F})+\epsilon_{p_i}(\mathcal{F}))E_i^*.$$
It suffices to prove that $[F^*]$ is not a linear combination of the elements in $V(\emptyset)$. Indeed, reasoning by contradiction, if $[F^*] =\gamma_0 [K_{\widetilde{\mathcal{F}}}-K_{S_{\mathcal{F}}}]+\sum_{j=1}^{n-d} \gamma_j [\widetilde{E}_{k_j}]$, with $\gamma_j\in \mathbb{R}$ for all $j=0,\ldots,n-d$, then, taking intersection product with $[F^*]$ at both sides of the equality, one has that $0 = \gamma_0(d_2 + 2)$; therefore $\gamma_0 = 0$ because $d_2 + 2 > 0$ \cite[Proposition 3.2]{GalMonOliv}. Now, taking intersection product with $[\hat{E}_j]$ we conclude that $\gamma_j = 0$ for all $j=1,\ldots,n-d$ leading to a contradiction. 

To prove the first statement in (d), assume that \fol is algebraically integrable and let $\df=aF^*+bM^*-\sum_{i=1}^nm_iE_i^*$ be the characteristic divisor of \fol. As before, we are going to show that $[F^*]$ is not a linear combination of the elements in $V(\Sigma)$. Again reasoning by contradiction, suppose that $[F^*]=\gamma_0 [K_{\widetilde{\mathcal{F}}}-K_{S_{\mathcal{F}}}]+\sum_{j=1}^{n-d} \gamma_j [\widetilde{E}_{k_j}]+\sum_{r=1}^{\sigma}\gamma'_r[\widetilde{C}_r]$, with $\gamma_j,\gamma'_r\in \mathbb{R}$ for all $j=0,\ldots,n-d,\,r=1,\ldots,\sigma$, then, taking intersection product with $[\df]$ at both sides of the equality, one gets that $b=0$ (by Lemma \ref{clave} and Remark \ref{K-K}), leading, by Proposition \ref{befe}, to a contradiction.

Finally, we prove the last statement in (d). Notice that by Lemma \ref{clave}, $D_{\mathcal{F}}\cdot \widetilde{E}_i=0$ (respectively, $D_{\mathcal{F}}\cdot \widetilde{C}=0$) if $E_i$ is non-dicritical (respectively, $C\in \Sigma$). Moreover, $D_{\mathcal{F}}\cdot (K_{\widetilde{\mathcal{F}}}-K_{Z_{\mathcal F}})=0$ (see Remark \ref{K-K}), and therefore $[\tf]\in V(\Sigma)^\perp$. The fact that $\tf \cdot F^*=1$ concludes the proof.
 
\end{proof}

Let \fol be the extension to $S_0$ of a polynomial foliation on $\C^2$. We are going to use the following notation:
\begin{equation}\label{haches}
{\everymath={\displaystyle}
\begin{array}{c}
h_0:=1,\,h_j:=0,\,1\leq j\leq d\,\text{ if }S_0=\PP^2,\,\text{ and otherwise (}S_0=\fd)\\[5pt]
h_0:=-\frac{d_1-\delta+2}{d_2+2}-\delta ,\, h_j:=\frac{\sum\limits_{i=1}^n(\nu_{p_i}(\mathcal F)+\epsilon_{p_i}(\mathcal F))\lambda_{ij}}{d_2+2},\, 1\leq j \leq d,
\end{array}}
\end{equation}
where the values $\lambda_{ij}$ were defined in Subsection \ref{rfi}. Notice that, when $S_0=\fd$, $h_0, h_1,\ldots,h_d$ are well-defined (because $d_2>-2$ \cite[Proposition 3.2.]{GalMonOliv}) and that they can be computed from the dicritical reduction of singularities of \fol. In addition, the values $\lambda_{ij}$ can be computed from the proximity graph of $\mathcal{B}_{\mathcal F}$ (see Remark \ref{lambdas}).

Our next lemma considers the case $\Sigma=\emptyset$.

\begin{lemma}\label{clave4}

Set $Y(S_0)=V(\emptyset)^\perp \cap [L^*]_{=1}$ (respectively, $\left(V(\emptyset)\setminus\{[K_{\widetilde{\mathcal F}}-K_{Z_\mathcal{F}}]\}\right)^\perp$ $\cap [F^*]_{=1}$) when $S_0=\PP^2$ (respectively, $\fd$). An element ${\bold x}\in NS(Z_{\mathcal F})$ belongs to $Y(S_0)$ if and only if there exists a vector ${\alpha}=(\alpha_1,\ldots,\alpha_d)\in \mathbb{R}^d$ such that 
{\small{\begin{equation}\label{eee}
{\everymath={\displaystyle}
\begin{array}{c}
{\bold x}={\bold v}({\alpha}):=[L^*]- \sum_{i=1}^n\left(\sum_{j=1}^d \lambda_{ij}\alpha_j\right) [E^*_{i}]\\
\left(\text{respectively, }{\bold x}={\bold v}({\alpha}):=\left(h_0+\sum_{j=1}^d h_j\alpha_j\right)[F^*]+[M^*]- \sum_{i=1}^n\left(\sum_{j=1}^d \lambda_{ij}\alpha_j\right) [E^*_{i}]\right),
\end{array}}
\end{equation}}}
whenever $S_0=\PP^2$ (respectively, $\fd$).
\end{lemma}

\begin{proof}

Let $W(S_0)$ be the affine subspace of $NS(Z_{\mathcal F})$ given by the set $\{{\mathbf v}(\alpha)\mid \alpha\in \mathbb{R}^d\}$, where ${\mathbf v}(\alpha)$ is the $\R$-divisor on $S_0$ defined in the statement. On the one hand, straightforward computations show that $W(S_0)\subseteq Y(S_0)$. On the other hand, a similar reasoning to that of the proof of Lemma \ref{clave3} proves that the dimension of the affine subspace $Y(S_0)$ is equal to $d$. This concludes the proof.

\end{proof}

Keep the previous notation, given a restricted set $\Sigma=\{C_1,\ldots,C_\sigma\}$ of independent algebraic solutions of \fol, by Lemma \ref{clave4} it holds that when $S_0=\PP^2$ (respectively, $\fd$) ${\bold x}\in V(\Sigma)^\perp \cap [G]_{=1}$, where $G=L^*$ (respectively, $F^*$), if and only if ${\bold x}={\bold v}(\alpha)$, for some solution $\alpha$ of the following system of $\sigma+1$ (respectively, $\sigma$) linear equations in $d$ real unknowns $\alpha_1,\ldots,\alpha_d$: 
{\small{ $${\bold v}({\alpha})\cdot [K_{\widetilde{\mathcal F}}-K_{Z_\mathcal{F}}]=0,{\bold v}({\alpha})\cdot [\widetilde{C}_i]=0,\,1\leq i\leq\sigma\text{ (respectively, }{\bold v}({\alpha})\cdot [\widetilde{C}_i]=0,\,1\leq i\leq\sigma).$$}}
  By Lemma \ref{clave3}, the dimension of $V(\Sigma)^\perp \cap [G]_{=1}$ (as an affine subspace) is  $\ell=d-\sigma-1$ when $S_0=\PP^2$ and $d-\sigma$ otherwise. Hence, using Gauss-Jordan elimination (and, possibly, reordering the infinitely near terminal dicritical singularities $p_{k_1},\ldots,p_{k_d}$), we conclude the existence of rational numbers $\mu_{k,s}$, $0\leq k\leq \ell$, $\ell+1\leq s\leq d$, such that the solution set of the mentioned system is the set
{\small{$$\Delta:=\left\{ \left(\alpha_1,\;\ldots, \alpha_\ell,\; \mu_{0,\ell+1}+\sum_{k=1}^\ell \mu_{k, \ell+1} \alpha_k, \ldots, \mu_{0,d}+\sum_{k=1}^\ell \mu_{k,d} \alpha_k\right)\mid \alpha_1,\ldots,\alpha_\ell \in \mathbb{R}\right\}.$$}}
Thus, with the given definition of $G$ depending on the considered surface $S_0$, we have deduced that
\begin{equation}\label{v}
V(\Sigma)^\perp\cap [G]_{=1}=\{{\bold v}(\alpha)\mid \alpha\in \Delta\}.
\end{equation}

With the above notation, especially that after (\ref{tt2}) and in (\ref{haches}), let us consider the following rational numbers:
\begin{equation}\label{eqn_Lmbd}
{\everymath={\displaystyle}
\begin{array}{ll}
\Lambda_{i0}&:=\sum_{s=\ell+1}^d \lambda_{is}\mu_{0,s}\text{, for } 1\leq i\leq n,\\
\Lambda_{ik}&:=\lambda_{ik}+\sum_{s=\ell+1}^d \lambda_{is}\mu_{k,s}\text{, for } 1\leq i\leq n, 1\leq k\leq \ell,\\
H_0&:=h_0+\sum_{s=\ell+1}^d h_s\mu_{0,s},\;\; \mbox{ and }\\
H_k&:=h_k+\sum_{s=\ell+1}^d h_s\mu_{k,s}\text{, for } 1\leq k\leq \ell.
\end{array}}
\end{equation}
Equalities (\ref{v}) and (\ref{epsilonpico}) lead to the following result, which gives an explicit description of $V(\Sigma)^\perp\cap [G^*]_{=1}$ in terms of free real parameters $\alpha_1,\ldots,\alpha_\ell$.

\begin{proposition}\label{39}
Let \fol be the extension to $S_0$ of a polynomial foliation on $\C^2$.  Let $\Sigma$ be a restricted set of independent algebraic solutions of ${\mathcal F}$ of length $\sigma$. Keep the above notation, in particular, set $G=L^*$ (respectively, $F^*$) and $\ell=d-\sigma-1$ (respectively, $d-\sigma$) when $S_0=\PP^2$ (respectively, $\fd$), $d$ being the cardinality of the set of terminal dicritical singularities. 

Then, an element ${\bold x}\in NS(Z_{{\mathcal F}})$ belongs to $V(\Sigma)^\perp\cap [G]_{=1}$ if and only if there exist $\alpha_1,\ldots,\alpha_\ell\in \mathbb{R}$ such that 
\begin{gather*}
{\bold x}=[L^*]-\sum_{i=1}^n \left(\Lambda_{i0}+\sum_{k=1}^\ell \Lambda_{ik}\alpha_k\right) [E_i^*]\,\\
\left(\text{respectively, }{\bold x}=\left(H_0+\sum_{k=1}^\ell H_k\alpha_k\right)[F^*]+[M^*]-\sum_{i=1}^n \left(\Lambda_{i0}+\sum_{k=1}^\ell \Lambda_{ik}\alpha_k\right) [E_i^*]\right).
\end{gather*}
\end{proposition}

The above proposition suggests to introduce the following family of $\R$-divisors depending on a parameter vector $\alpha=(\alpha_1,\ldots,\alpha_\ell)\in \mathbb{R}^\ell$:
\begin{equation}\label{eqn_Talfa}
{\everymath={\displaystyle}
\begin{array}{c}
T_{\alpha}:=L^*-\sum_{i=1}^n \left(\Lambda_{i0}+\sum_{k=1}^\ell \Lambda_{ik}\alpha_k\right) E_i^*,\,\text{ if }S_0=\PP^2\text{ and}\\
T_{\alpha}:=\left(H_0+\sum_{k=1}^\ell H_k\alpha_k\right)F^*+M^*-\sum_{i=1}^n \left(\Lambda_{i0}+\sum_{k=1}^\ell \Lambda_{ik}\alpha_k\right) E_i^*,\,\text{ if }S_0=\fd.
\end{array}}
\end{equation}
The following result keeps the notation as in Proposition \ref{39} and shows that the above family of $\R$-divisors contains an interesting distinguished element.

\begin{theorem}\label{teo1}
Assume that ${\mathcal F}$ is algebraically integrable, let $\Sigma$ be a restricted set of independent algebraic solutions of ${\mathcal F}$ and consider the $\Q$-divisor $\tf$ defined in (\ref{eqn_def_tf}). Then 
\begin{itemize}
\item[(a)] There exists a specific vector $\alpha=(\alpha_1,\ldots,\alpha_\ell)$ such that $\alpha_i$ is a strictly positive rational number for all $i=1,\ldots,\ell$, and $T_{{\mathcal F}}=T_{\alpha}$.
\item[(b)] Moreover, for such an $\alpha$,
\begin{multline*}
T_{\alpha}^2=-\sum_{k,k'=1}^\ell\left(\sum_{i=1}^n \Lambda_{ik}\Lambda_{ik'}\right)\alpha_k\alpha_{k'} -2 \sum_{k=1}^\ell \left(\sum_{i=1}^n\Lambda_{i0}\Lambda_{ik}\right)\alpha_k-\sum_{i=1}^n\Lambda_{i0}^2+1=0\\
\text{ if }S_0=\PP^2,
\end{multline*}
\begin{multline*}
T_{\alpha}^2=-\sum_{k,k'=1}^\ell\left(\sum_{i=1}^n \Lambda_{ik}\Lambda_{ik'}\right)\alpha_k\alpha_{k'} + \sum_{k=1}^\ell \left(2H_k-2\left(\sum_{i=1}^n\Lambda_{i0}\Lambda_{ik}\right)\right)\alpha_k+2H_0\\
-\sum_{i=1}^n\Lambda_{i0}^2+\delta=0\text{ if }S_0=\fd.
\end{multline*}
\end{itemize}

\end{theorem}

\begin{proof}
Let us prove Part (a). Since $\tf\in \left(V(\emptyset)\setminus\{[K_{\widetilde{\mathcal F}}-K_{Z_\mathcal{F}}]\}\right)^\perp \cap [L^*]_{=1}$ (respectively, $V(\emptyset)\cap[F^*]_{=1}$) when $S_0=\PP^2$ (respectively, $\fd$), by Lemma \ref{clave4}, there exist $\alpha_1,\ldots,\alpha_d\in \mathbb{R}$ such that $T_{{\mathcal F}}={\bold v}(\alpha_1,\ldots,\alpha_d)$ (see Lemma \ref{clave4}). These values $\alpha_1,\ldots,\alpha_d$ coincide with the values $\beta_1,\ldots,\beta_d$ in Equality (\ref{tt2}) and, therefore, they are strictly positive rational numbers. Now, Proposition \ref{39} and the paragraph after the proof of Lemma \ref{clave4} show that, after reordering (if necessary) the infinitely near dicritical singularities (and, consequently, the values $\alpha_1,\ldots,\alpha_d$), it holds the equality of divisors $T_{{\mathcal F}}=T_\alpha$, where $\alpha=(\alpha_1,\ldots,\alpha_\ell)$ and Part (a) is proved.

Part (b) follows from by Lemma \ref{clave2}.
    
\end{proof}

Keep the above notation where we have a foliation \fol and a restricted set $\Sigma$ of independent algebraic solutions. $\Sigma$ allows us to compute the values in (\ref{eqn_Lmbd}). Then consider the following system of linear equations (whose unknowns are $\alpha_1,\ldots,\alpha_\ell$):
\begin{equation}\label{system1}
{\everymath={\displaystyle}
\left\lbrace\begin{array}{c}
\sum_{k= 1}^\ell \left(\sn \Lambda_{i1}\Lambda_{ik}\right)\alpha_k=H_1-\left(\sum_{i=1}^n\Lambda_{i0}\Lambda_{i1}\right)\\
\sum_{k=1}^\ell \left(\sn \Lambda_{i2}\Lambda_{ik}\right)\alpha_k=H_2-\left(\sum_{i=1}^n\Lambda_{i0}\Lambda_{i2}\right)\\
\vdots\\
\sum_{k=1}^\ell \left(\sn \Lambda_{i\ell}\Lambda_{ik}\right)\alpha_k=H_\ell-\left(\sum_{i=1}^n\Lambda_{i0}\Lambda_{i\ell}\right).
\end{array}\right.}
\end{equation}
The coefficient matrix of (\ref{system1}) is the Gram matrix $\mathfrak{G}$ of the set of vectors $\{(\Lambda_{1j},\ldots,\Lambda_{nj})\}_{j=1}^\ell\subseteq \mathbb{R}^\ell$ with respect to the Euclidean inner product. These vectors are linearly independent and, therefore, $\mathfrak{G}$ is a positive definite matrix. In particular, \emph{System (\ref{system1}) has a unique solution $\alpha_{{\mathcal F}}^\Sigma$}. Notice that this value is a rational vector.

Set $h:\mathbb{R}^\ell\rightarrow \mathbb{R}$ the map defined by $h[\alpha=(\alpha_1,\ldots,\alpha_\ell)]:=T_{\alpha}^2$. It satisfies the following result.

\begin{lemma}\label{lemm}
The map $h$ has an absolute maximum, which is only reached at $\alpha_{{\mathcal F}}^\Sigma$.
\end{lemma}

\begin{proof}
The map $h$ is a sum of an affine map and a negative definite quadratic form whose matrix is $-\mathfrak{G}$, where $\mathfrak{G}$ is the Gram matrix given by System (\ref{system1}). Thus, $h$ is bounded from above. The Jacobian vector and the Hessian matrix of $h$ (which is $-2\mathfrak{G}$) show that $h$ has a local, and hence absolute, maximum, only reached at $\alpha_{{\mathcal F}}^\Sigma$.
    
\end{proof}

Now we state the main result in this section, which supports our algorithms for deciding about algebraic integrability.

\begin{theorem}\label{teo2}
Let $\Sigma$ be a restricted set of independent algebraic solutions of the extension \fol to $S_0$ of a polynomial foliation on $\C^2$. Set $T_\alpha$ the divisor depending on a parameter vector $\alpha=(\alpha_1,\ldots,\alpha_\ell)$ introduced in (\ref{eqn_Talfa}) and keep the above notation.
\begin{itemize}
    \item[(a)] If  $T^2_{\alpha_{{\mathcal F}}^\Sigma}<0$, then \fol is not algebraically integrable.
    \item [(b)] If $T^2_{\alpha_{{\mathcal F}}^\Sigma}=0$ and \fol is algebraically integrable, then the divisor $\tf$ introduced in (\ref{eqn_def_tf}) satisfies $T_{{\mathcal F}}=T_{\alpha_{{\mathcal F}}^\Sigma}$ and  $\alpha_{{\mathcal F}}^\Sigma\in (\mathbb{Q}_{>0})^\ell$.
\end{itemize}
\end{theorem}

\begin{proof}
Part (a) follows by Lemma \ref{lemm} and Theorem \ref{teo1}. 

To prove (b), by Lemma \ref{lemm}, $T_{\alpha}^2=0$ if and only if $\alpha=\alpha_{{\mathcal F}}^\Sigma$, and the result follows by Theorem \ref{teo1}.
    
\end{proof}

Let \fol be the extension to $S_0$ of a foliation on the complex plane and $\zf$ the surface obtained after reducing the dicritical singularities of \fol (see Subsection \ref{rfi}). For each $\mathbb{Q}$-divisor $D$ on $Z_{\mathcal{F}}$, let $\mathcal{R}(D):=\{m\in \mathbb{Z}_{>0}\mid mD \mbox{ is a divisor}\}$. Define $e(D):=0$ if $\dim |mD|<1$ for every $m\in \mathcal{R}(D)$ and, otherwise,
\begin{equation}\label{eqn_eD}
e(D):=\min \{m\in \mathcal{R}(D) \mid \dim |mD|\geq 1\},
\end{equation}
where $\dim$ stands for projective dimension.

The next result relates the divisor $\df$ introduced in (\ref{divisor}) to the divisor $\tf$ whenever \fol is algebraically integrable.

\begin{proposition}\label{conditionC}
If \fol is algebraically integrable, then $D_{\mathcal{F}}=e(T_{\mathcal{F}})T_{\mathcal{F}}$.
\end{proposition}

\begin{proof}
  Let $b$ be a positive integer such that $D_{\mathcal{F}}=bT_{\mathcal{F}}$. Since $\dim |D_{\mathcal{F}}|=1$, it holds that $b\geq  e(T_{\mathcal{F}})$. 
  
To conclude the proof and reasoning by contradiction, assume that $e(T_{\mathcal{F}})<b$. Then, any curve $C\in |e(T_{\mathcal{F}})T_{\mathcal{F}}|$ satisfies $D_{\mathcal{F}}\cdot C=0$ and, therefore, it is $\widetilde{\mathcal{F}}$-invariant (by Lemma \ref{clave}). This implies that there are infinitely many $\widetilde{\mathcal{F}}$-invariant reduced and irreducible curves which are not linearly equivalent to $D_{\mathcal{F}}$, which is a contradiction.

\end{proof}

The following algorithm decides about the existence of a rational first integral of \fol (and computes it in the affirmative case). It runs under certain assumptions that depend on the vector $\alpha_{\mathcal F}^\Sigma$ introduced before Lemma \ref{lemm} and concerns the attached divisor $T_{\alpha_{\mathcal F}^\Sigma}$ and value $e(T_{\alpha_{\mathcal F}^\Sigma})$.

\begin{algorithm}\label{algoritmo2}

$ $\vspace{0.25cm}

\textbf{Input:} A differential $1$-form $\Omega$ defining a foliation \fol on $S_0$, a restricted set of independent algebraic solutions $\Sigma$, the dicritical configuration ${\mathcal B}_{{\mathcal F}}$, the vector $\alpha_{\mathcal F}^\Sigma\in\R^\ell$ and the attached $\Q$-divisor $T_{\alpha_{\mathcal F}^\Sigma}$ satisfying at least one of the following conditions:
\begin{itemize}
    
    \item[(a)] $T^2_{\alpha_{{\mathcal F}}^\Sigma}<0$.
    \item[(b)] $T^2_{\alpha_{{\mathcal F}}^\Sigma}=0$ and $\alpha_{{\mathcal F}}^\Sigma\not\in (\mathbb{Q}_{>0})^\ell$.
    \item[(c)] $T^2_{\alpha_{{\mathcal F}}^\Sigma}=0$ and $K_{Z_{{\mathcal F}}}\cdot T_{\alpha_{{\mathcal F}}^\Sigma}<0$.
    \item[(d)] $T^2_{\alpha_{{\mathcal F}}^\Sigma}=0$ and $e(T_{\alpha_{{\mathcal F}}^\Sigma})>0$.
    \item[(e)] $e(T_{\alpha_{{\mathcal F}}^\Sigma})=0$.
\end{itemize}

\textbf{Output:} A rational first integral if \fol is algebraically integrable, and $0$ otherwise. \\

\begin{enumerate} [label=(\arabic*)]

\item If Conditions (a), (b) or (e) are satisfied, then return $0$. 
\item If Condition (c) is satisfied and $-2/(K_{Z_{{\mathcal F}}}\cdot T_{\alpha_{{\mathcal F}}^\Sigma})\in \mathcal{R}(T_{\alpha_{{\mathcal F}}^\Sigma})$ --defined after Theorem \ref{teo2}-- then, let $\gamma:=-2/(K_{Z_{{\mathcal F}}}\cdot T_{\alpha_{{\mathcal F}}^\Sigma})$. Otherwise (that is, Condition (d) holds) let $\gamma:=e(T_{\alpha_{{\mathcal F}}^\Sigma})$.
\item[(3)] Compute the linear system $|\gamma T_{\alpha_{{\mathcal F}}^\Sigma}|$. If it is not base point free or its (projective) dimension is not $1$, then return $0$. Otherwise, compute the equations of two curves on $S_0$, $F=0$ and $G=0$, corresponding to a basis of $(\pi_{\mathcal{F}})_*|\gamma T_{\alpha_{{\mathcal F}}^\Sigma}|$ (recall that $\pi_{\mathcal{F}}$ was defined in Subsection \ref{rfi}) and compute $\Omega\wedge (FdG-GdF)$. If the last result is $0$, then return $F/G$; otherwise return $0$.

\end{enumerate}
\end{algorithm}

\begin{remark}
Conditions (a), (b) and (c) are easily verifiable (once the dicritical reduction of singularities has been computed). However, we do not know a general effective characterization neither for Condition (e) nor for the second part of Condition (d). 
\end{remark}

\begin{remark}
If Condition (c) holds and \fol is algebraically integrable, then the genus of a rational first integral is 0 (see the forthcoming justification).
\end{remark}

\noindent \emph{Justification of Algorithm \ref{algoritmo2}}. Step (1) is justified by Theorem \ref{teo2} and Proposition \ref{conditionC}. 

Assume that Condition (c) holds. If \fol is algebraically integrable, then, by Part (b) of Theorem \ref{teo2}, the $\mathbb{Q}$-divisor $T_{\mathcal{F}}$ must coincide with $T_{\alpha_{{\mathcal F}}^\Sigma}$. Moreover there exists a positive integer $b$ such that $D_{\mathcal{F}}=b T_{\mathcal{F}}$. By Bertini's Theorem, the general elements of $|D_{\mathcal{F}}|$ are non-singular. Therefore, by Part (a) of Lemma \ref{clave2} and the adjunction formula, $1+\frac{b}{2}K_{Z_{{\mathcal F}}}\cdot T_{\mathcal{F}} =g$, where $g$ is the genus of the rational first integral. Since $K_{Z_{{\mathcal F}}}\cdot T_{\mathcal{F}}<0$, we conclude that $g=0$ and $b=\gamma:=-2/(K_{Z_{{\mathcal F}}}\cdot T_{\mathcal{F}})$. This fact and Proposition \ref{conditionC} justify Steps (2) and (3).

\begin{flushright}
\qedsymbol{}
\end{flushright}

The next result shows that, if one knows a priori that \fol is algebraically integrable and one has a restricted set of independent algebraic solutions $\Sigma$ such that $[D_{\mathcal{F}}]$ belongs to the linear span of $V(\Sigma)$, then Algorithm \ref{algoritmo2} returns a rational first integral of \fol.

\begin{proposition}\label{hhh}
Let \fol be an algebraically integrable foliation on $S_0$. Assume that $\mathcal{F}$ admits a restricted set $\Sigma$ of independent algebraic solutions such that the class $[D_{\mathcal{F}}]$ belongs to the linear span of $V(\Sigma)$. Then, Condition (d) in Algorithm \ref{algoritmo2} holds, and therefore, this algorithm returns a rational first integral of $\mathcal{F}$.
\end{proposition}

\begin{proof}
By Proposition \ref{39} and Lemma \ref{lemm} (with the notation as in these results) the self-intersection $T_{\alpha_{{\mathcal F}}^\Sigma}^2$ is the maximum of the set $R:=\{\mathbf{x}^2\mid \mathbf{x}\in V(\Sigma)^\perp\cap  [G]_{=1}\}$, where $G=L^*$ (respectively, $F^*$) when $S_0=\PP^2$ (respectively, $\fd$). 

Now, $V(\Sigma)^\perp\subseteq [D_{\mathcal{F}}]^\perp$ since $[D_{\mathcal{F}}]$ belongs to the linear span of $V(\Sigma)$.
Moreover, any element of the hyperplane $[D_{\mathcal{F}}]^\perp$ has non-positive self-intersection (because $D_{\mathcal{F}}$ is a nef divisor). Finally, $[T_{\mathcal{F}}]\in V(\Sigma)^\perp\cap [G]_{=1}$ and $T_{\mathcal{F}}^2=0$ (by Lemma \ref{clave2}). As a consequence, $0$ belongs to $R$ and therefore $T_{\alpha_{{\mathcal F}}^\Sigma}^2=0$. This equality and the fact that, by Proposition \ref{conditionC}, $e(T_{\alpha_{{\mathcal F}}^\Sigma})>0$, conclude the proof.

\end{proof}

The following result shows that the computation of the integral components of a fiber of the pencil ${\mathcal P}_{\mathcal{F}}$ introduced in Subsection \ref{rfi} leads us to obtain a rational first integral.

\begin{corollary}\label{remmm}
Assume that $\mathcal{F}$ is algebraically integrable and let $\Sigma'$ be a finite set of integral $\mathcal{F}$-invariant curves containing all the integral components of a curve of the pencil ${\mathcal P}_{\mathcal{F}}$. Then
\begin{itemize}
\item[(a)] If $\Sigma\subseteq \Sigma'$ is such that $V(\Sigma)$ is a basis of the linear span of $V(\Sigma')$, then $\Sigma$ is a restricted set of independent algebraic solutions (of $\mathcal{F}$) and $[D_{\mathcal{F}}]$ belongs to the linear span of $V(\Sigma)$.
\item[(b)] For any subset $\Sigma\subseteq \Sigma'$ satisfying the condition given in (a), Algorithm \ref{algoritmo2} (applied to $\mathcal{F}$, the dicritical configuration $\mathcal{B}_{\mathcal{F}}$ and $\Sigma$) returns a rational first integral of $\mathcal{F}$.
\end{itemize}

\end{corollary}
\begin{proof}
Let $n=\#\mathcal{B}_{\mathcal{F}}$. Firstly we prove Part (a). Clearly $\Sigma$ is a restricted set of independent algebraic solutions. The curve of the pencil ${\mathcal P}_{\mathcal{F}}$ whose integral components are in $\Sigma'$ corresponds to a fiber $G$ of the morphism $\varphi: Z_\mathcal{F}\rightarrow \mathbb{P}^1$ induced by the complete linear system $|D_{\mathcal{F}}|$ and, then, $G$ has the form \[
\sum_{C\in \Sigma'}a_C C+\sum_i b_i \widetilde{E}_i,\]
where the indices $i$ of the second summand run over those indices in the set $\{1,\ldots,n\}$ such that the exceptional divisor $E_i$ is non-dicritical. Moreover, $a_C$ and $b_i$ are not negative. Since $D_{\mathcal{F}}$ and $G$ are linearly equivalent, it holds that $[D_{\mathcal{F}}]$ belongs to the linear span of $V(\Sigma)$.

Part (b) follows from Part (a) and Proposition \ref{hhh}.

\end{proof}

We are ready to state our \emph{main algorithm}. It decides whether a foliation \fol on $S_0$ (recall that it is  either the projective plane or a Hirzebruch surface) has a rational first integral of genus $g\neq 1$ and compute it in the affirmative case. We warn that under a specific condition we will state, the algorithm may not finish, giving an ambiguous output. Later we will show that a slight modification of our algorithm always decide if there is a polynomial first integral of genus $g\neq 1$ of a polynomial foliation of the complex plane.

\begin{algorithm}\label{algoritmo3}

$ $\vspace{0.25cm}

\textbf{Input:} A differential $1$-form $\Omega$ defining a foliation ${\mathcal F}$ on $S_0$, a restricted set of independent algebraic solutions $\Sigma$, the dicritical configuration ${\mathcal B}_{{\mathcal F}}$, the $\Q$-divisor $T_{\alpha_{{\mathcal F}}^\Sigma}$ (see the definition below Proposition \ref{39} and Lemma \ref{lemm}) and a non-negative integer $g\neq 1$.\\

\textbf{Output:} Either a rational first integral of genus $g$ of ${\mathcal F}$, or $0$ (what means that $\mathcal{F}$ has no rational first integral of genus $g$), or $-1$ (what means that we are not able to decide about the algebraic integrability of \fol).\\

\begin{enumerate} [label=(\arabic*)]
\item If $T_{\alpha_{{\mathcal F}}^\Sigma}^2<0$, then return $0$. 
\item If $T_{\alpha_{{\mathcal F}}^\Sigma}^2=0$ and $K_{Z_{\mathcal{F}}}\cdot T_{\alpha_{{\mathcal F}}^\Sigma}=0$, then return $0$. This condition is equivalent to the fact that, in the case of algebraic integrability of \fol, the genus of a \ipr of \fol is $1$.
\item If $T_{\alpha_{{\mathcal F}}^\Sigma}^2=0$ and $K_{Z_{\mathcal{F}}}\cdot T_{\alpha_{{\mathcal F}}^\Sigma}\neq 0$, then compute $\gamma:=2(g-1)/(K_{Z_{{\mathcal F}}}\cdot T_{\alpha_{{\mathcal F}}^\Sigma})$ and perform the following steps:
\begin{enumerate}[label=(\arabic{enumi}.\arabic{enumii})]
    \item If $\gamma\not\in \mathcal{R}(T_{\alpha_{{\mathcal F}}^\Sigma})$, then return $0$.
    \item If the linear system $|\gamma T_{\alpha_{{\mathcal F}}^\Sigma}|$ has (projective) dimension different from $1$ or it is not base point free, then return $0$. Otherwise, compute the equations of two curves on $S_0$, $F=0$ and $G=0$, corresponding to a basis of $(\pi_{\mathcal{F}})_*|\gamma T_{\alpha_{{\mathcal F}}^\Sigma}|$ ($\pi_{\mathcal{F}}$ was defined in (\ref{www})), and compute $\Omega\wedge (FdG-GdF)$. If the last result is $0$, then return $F/G$; otherwise return $0$.
\end{enumerate}
\item If $T_{\alpha_{{\mathcal F}}^\Sigma}^2>0$, then consider the set of $\ell$-tuples 
\begin{equation}\label{eqn_Delta}
\mathfrak{D}=\{\alpha\in (\mathbb{Q}_{\geq 0})^\ell\mid T_{\alpha}^2=0\}
\end{equation}
 and compute the real values $p_{\inf}:=\inf \{K_{\mathcal{F}}\cdot T_{\alpha}\mid \alpha\in \mathfrak{D}\}$ and $p_{\sup}:=\sup \{K_{\mathcal{F}}\cdot T_{\alpha}\mid \alpha\in \mathfrak{D}\}$. If $p_{\inf}\cdot p_{\sup}\leq 0$, then return $-1$. Otherwise:
\begin{enumerate}[label=(\arabic{enumi}.\arabic{enumii})]
\item If ($p_{\sup}<0$ and $g\neq 0$) or ($p_{\inf}>0$ and $g=0$), then return $0$.
\item If $p_{\sup}<0$, then consider the set of integers $V:=[\frac{-2}{p_{\inf}},\frac{-2}{p_{\sup}}]\cap \mathbb{Z}$. 
\item If $p_{\inf}>0$, then consider the set of integers $V:=[\frac{2(g-1)}{p_{\sup}},\frac{2(g-1)}{p_{\inf}}]\cap \mathbb{Z}$.
\item If $V=\emptyset$, then return $0$. Otherwise, let ${\rm pr}_k: (\mathbb{Q}_{\geq 0})^\ell \rightarrow \mathbb{Q}_{\geq 0}$ be the projection map onto the $k$th coordinate, $1\leq k\leq l$, and compute two non-negative rational numbers $\alpha_k^{-}$ and $\alpha_k^{+}$ such that\[
{\rm pr}_k(\mathfrak{D})\subseteq [\alpha_k^{-},\alpha_k^{+}].\]
 Also, for each $t\in V$, consider the finite set 
$$A_t:=\bigcap_{k=1}^\ell {\rm pr}_k^{-1}([t\alpha_k^{-},t\alpha_k^{+}]\cap \mathbb{Z}).$$

\item For each $t\in V$ and for each $s\in A_t$, check whether:
\begin{enumerate}[label=(\arabic{enumi}.\arabic{enumii}.\arabic{enumiii})]
\item $tT_{s/t}$ is a divisor on $Z_{\mathcal F}$, $T_{s/t}^2=0$ and $|tT_{s/t}|$ is a base point free linear system of (projective) dimension $1$. In the affirmative case, compute the equations $F=0$ and $G=0$ of two curves on $S_0$ corresponding to a basis of $(\pi_{\mathcal{F}})_*|tT_{s/t}|$ and verify whether $\Omega\wedge (FdG-GdF)$ vanishes. In the positive case, return $F/G$.
 \end{enumerate}
 \item Return $0$.
\end{enumerate}
\end{enumerate}
\end{algorithm}

\noindent \emph{Justification of Algorithm \ref{algoritmo3}}. Step (1) is justified by Part (a) of Theorem \ref{teo2} while Step (2) is justified by Part (b) of the same theorem and the adjunction formula (since $g\neq 1$). Moreover, Part (b) of Theorem \ref{teo2}, Proposition \ref{conditionC} and Lemma \ref{clave2} justify Step (3).

In order to justify Step (4), assume that $T_{\alpha_{{\mathcal F}}^\Sigma}^2>0$. Notice that the set $\mathfrak{D}$ is non-empty and bounded and, therefore, $p_{inf}$ and $p_{sup}$ are well-defined and can be computed. For instance, by using Lagrange multipliers to obtain the extrema of the map $f:\mathbb{R}^\ell\rightarrow \mathbb{R}$ defined by $f(\alpha_1,\ldots,\alpha_\ell):=K_{\mathcal{F}}\cdot T_{\alpha^2}$ subject to the restriction $T_{\alpha^2}^2=0$, where $\alpha:=(\alpha_1,\ldots,\alpha_\ell)$ and $\alpha^2:=(\alpha_1^2,\ldots,\alpha_\ell^2)$. To adapt $f$ to our specific requirements, we have introduced a slight modification by using $\alpha^2$ instead of $\alpha$, as would be natural. This adaptation is motivated by our interest in obtaining solutions within the domain of positive numbers. While Lagrange multipliers are conventionally used for real-valued functions, working with $\alpha^2:=(\alpha_1^2,\ldots,\alpha_\ell^2)$ allows us to ensure that the restriction remains applicable to positive values.

This approach leads us to an efficient computation of a sufficiently accurate approximation of the values $p_{\inf}$ and $p_{\sup}$. Notice that, for our purposes, it suffices good approximations of $p_{\inf}$ and $p_{\sup}$.

Assume now that $\mathcal{F}$ has a rational first integral of genus $g$. Let $\beta\in (\mathbb{Q}_{\geq 0})^\ell$ such that $T_{\mathcal{F}}=T_{\beta}$ and let $e(T_{\beta})$ be the integer defined in (\ref{eqn_eD}). Notice that $e(T_{\beta})$ is equal to the coefficient $b$ in (\ref{divisor}) and it satisfies $[D_{\mathcal{F}}]=[b T_{\beta}]$). Since $\beta\in\mathfrak{D}$,
\begin{equation}\label{eqn_pi_ps}
b\,p_{\inf}\leq 2g-2\leq b\,p_{\sup}
\end{equation}
 by the adjunction formula. This proves that if $p_{\sup}<0$ (respectively, $p_{\inf}>0$), $g=0$ (respectively, $g\neq 0$) which justifies Step (4.1).

If $p_{\sup}<0$ and $g=0$ (respectively, $p_{\inf}>0$ and $g\neq 0$), Inequalities (\ref{eqn_pi_ps}) show that $b=t$ for some $t\in V$. Moreover, $b\beta$ belongs to the set $A_b$ defined in the algorithm; therefore, the divisor $D_{\mathcal{F}}$ (introduced in (\ref{divisor})) equals $b T_{s/b}$ for some $s\in A_b$. These facts and Lemma \ref{clave2} show that Step (4) works. It is convenient to add that the bounds $\alpha_k^-$ and $\alpha_k^+$ ($k=1,\ldots,\ell$) can be computed with the help of similar procedures to those used to compute $p_{\sup}$ and $p_{\inf}$.
\begin{flushright}
\qedsymbol{}
\end{flushright}

\begin{remark}
With the above notation, assume that $p_{\inf}\cdot p_{\sup}\leq 0$. Then, the inequalities in (\ref{eqn_pi_ps}) provide two lower bounds for $b$; let us denote by $b_{\max}$ the largest one. Then, it means that $b\in [b_{\max},\infty)\cap \mathbb{Z}$, which is not a finite set. This is the reason why the algorithm returns $-1$ in this case.\end{remark}

One can avoid the output $-1$ in Algorithm \ref{algoritmo3} when looking for specific types of rational first integrals. We explain these cases in the following two results.
\begin{remark}\label{rr1}
Fix $r\geq 1$ irreducible polynomials $f_1,\ldots,f_r$ in $\mathbb{C}[x,y]$. A simplified form of Algorithm \ref{algoritmo3} allows us to decide whether a polynomial foliation \folc on $\mathbb{C}^2$ has, or not, a rational first integral of genus $g\neq 1$ of the form 
\begin{equation}\label{type}
\frac{f(x,y)}{{f_1(x,y)}^{a_1}\cdots {f_r(x,y)}^{a_r}},
\end{equation}
where $f\in \mathbb{C}[x,y]$ and $a_s$, $1\leq s\leq r$, are positive integers.

Indeed, let $C_s$ be the closure of the image of the affine curve with equation $f_s(x,y)=0$ by the inclusion $U_Z\rightarrow \PP^2$ (respectively, $U_{00}\rightarrow \mathbb{F}_{\delta}$), after identifying the affine plane with $U_Z$ (respectively, $U_{00}$). Consider a maximal (with respect to the inclusion) restricted set $\Sigma$ of independent algebraic solutions of the extension \fol of \folc to $\PP^2$ (respectively, $\fd$) contained in $\{C_1,\ldots,C_r\}$ (respectively, $\{C_1,\ldots,C_r,C_{X_0},C_{Y_0}\}$).

In Algorithm \ref{algoritmo3}, replace Step (4) by\medskip
\begin{center}
(4') If $T_{\alpha_{{\mathcal F}}^\Sigma}^2>0$, then return $0$.\medskip 
\end{center}
This modified algorithm, where the input is the same as that in Algorithm \ref{algoritmo3} but $\Sigma$ must be as above, gives an output which is either a rational first integral of \fol of genus $g$, or $0$ (that means that \folc has no rational first integral of genus g and type (\ref{type})).

The algorithm runs because if \folc has a rational first integral of type (\ref{type}), then, by Part (a) of Corollary \ref{remmm}, the class $[D_{\mathcal{F}}]$ belongs to the linear span of $V(\Sigma)$ and hence $T_{\alpha^\Sigma_\mathcal{F}}^2=0$ by Proposition \ref{hhh}.

\end{remark}

We conclude this section with an interesting corollary whose proof follows from Remark \ref{rr1}.

\begin{corollary}\label{rr2}
Let \folc be a polynomial foliation on $\C^2$. Then, the algorithm described in Remark \ref{rr1}, with $r=1$ and $f_1(x,y)=1$, decides whether or not \folc has a polynomial first integral of fixed genus $g\neq 1$. In the affirmative case, it computes the first integral.

\end{corollary}

\section{Examples}\label{sec_examples}

To finish this paper we supply some examples which explain how our algorithms work. We hope this also helps the reader to understand our results and algorithms.

Let us denote by $C_{f}$ the curve on $S_0$, which can be either the projective plane $\PP^2$ or a Hirzebruch surface $\mathbb{F}_{\delta}$ with equation $f=0$.

\begin{example}\label{ex_noip}
Let \folc be the polynomial foliation on $\mathbb{C}^2$ defined by the following $1$-form:
 \begin{multline*}
\omega:=(-8 y + 9 x^2 y + 3 y^3 - 3 x^2 y^3)dx+(8 x - 3 x^3 - 9 x y^2 + 3 x^3 y^2 - 2 y^3)dy.\end{multline*}
 Consider its extended foliation $\mathcal{F}^1$ to $\mathbb{F}_1$ given by the output of Algorithm \ref{algoritmo} with inputs $\omega$ and $\delta=1$. Its canonical sheaf is $\mathcal{K}_{\mathcal{F}^1}=\mathcal{O}_{\mathbb{F}_1}(2,3)$ and its dicritical configuration ${\mathcal B}_{{\mathcal F}^1}$ consists of 5 points, $p_1,\ldots,p_5$, such that $p_1,p_3,p_4,p_5\in \mathbb{F}_1$ and $p_2$ is infinitely near $p_1$. Moreover, $p_2,\,p_3,\,p_4$ and $p_5$ are the infinitely near terminal dicritical singularities, thus $d=4$. $\Sigma'=\{C_{X_0},C_{Y_0},C_{Y_1}\}$ is a set of $\mathcal{F}^1$-invariant curves. From the proximity relations among the points of ${\mathcal B}_{{\mathcal F}^1}$ and the equalities
$$[K_{\widetilde{\mathcal F}^1}-K_{Z_{{\mathcal F}^1}}]=3[F^*] +5[M^*] -4[E_1^*] -3[E_2^*] -2[E_3^*] -2[E_4^*] -2[E_5^*],$$
 $$[\widetilde{C}_{X_0}]=[F^*] -[E_1^*] \mbox{, } 
[\widetilde{C}_{Y_0}]=[M^*] -\sum_{i=1}^2[E_i^*] \mbox{ and } 
[\widetilde{C}_{Y_1}]=-[F^*]+[M^*] -\sum_{i=3}^5[E_i^*],$$
we can deduce that $\Sigma=\{C_{X_0},C_{Y_0}\}\subset \Sigma'$ is a restricted set of independent algebraic solutions of $\mathcal{F}^1$ of length $\sigma=2$ and $\Sigma'$ is not. Then, $\ell=d-\sigma=2$.

Considering parameters $\alpha_1,\alpha_2,\alpha_3$ and $\alpha_4$ associated, respectively, to the infinitely near terminal dicritical singularities $p_3, p_4, p_5$ and $p_2$ and expressing $\alpha_3$ and $\alpha_4$ in terms of of $\alpha_1$ and $\alpha_2$ (as explained before Proposition \ref{39}), one gets that $\alpha_{{\mathcal F}^1}^\Sigma=(1,1)$ and
$$T_{\alpha_{{\mathcal F}^1}^\Sigma}=[F^*] +[M^*] -[E_1^*] -[E_2^*] - [E_3^*] - [E_4^*] - [E_5^*].$$

Since $T^2_{\alpha^\Sigma_{\mathcal{F}^1}}=-2<0$, by Algorithm \ref{algoritmo2} we conclude that ${\mathcal F}^1$, and hence \fol, is not algebraically integrable.
\end{example}

\begin{example}\label{ex_ip1}
Consider now the polynomial foliation \folc on $\C^2$ defined by the $1$-form \begin{multline*}
\omega:=(x^4 - x^3 y + x^4 y^3 + 5 x^3 y^4 + 9 x^2 y^5 + 7 x y^6 + 2 y^7)dx+\\
(2 x^4 - 3 x^5 y^2 - 13 x^4 y^3 - 21 x^3 y^4 - 15 x^2 y^5 - 4 x y^6)dy.\end{multline*}
Take the extended foliation $\mathcal{F}^2$ of \folc to $\mathbb{F}_2$ given by the output of Algorithm \ref{algoritmo}, where the input is $(2,\omega)$. 

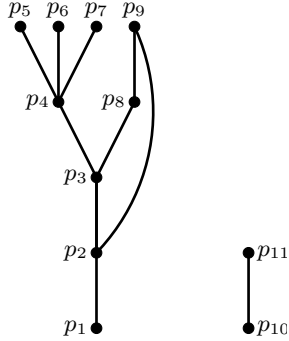
\begin{figure}[H]

\begin{center}
\definecolor{ccqqqq}{rgb}{0.8,0,0} 
\definecolor{xdxdff}{rgb}{0.49019607843137253,0.49019607843137253,1} 
\definecolor{ududff}{rgb}{0.30196078431372547,0.30196078431372547,1} 
\begin{tikzpicture}[line cap=round,line join=round,>=triangle 45,x=1cm,y=1cm]
\clip(1.5,-1.5) rectangle (5.5,3.5);
\draw [line width=1pt] (3,-1)-- (3,1);
\draw [line width=1pt] (3,0)-- (3,1);
\draw [line width=1pt] (3,1)-- (2.5,2);
\draw [line width=1pt] (3,1)-- (3.5,2);
\draw [line width=1pt] (2.5,2)-- (2,3);
\draw [line width=1pt] (2.5,2)-- (2.5,3);
\draw [line width=1pt] (2.5,2)-- (3,3);
\draw [line width=1pt] (3.5,3)-- (3.5,2);
\draw [shift={(1.031621621621621,1.8697297297297297)},line width=1pt]  plot[domain=-0.7597014056253908:0.4294040507961376,variable=\t]({1*2.7148485598100502*cos(\t r)+0*2.7148485598100502*sin(\t r)},{0*2.7148485598100502*cos(\t r)+1*2.7148485598100502*sin(\t r)});
\draw [line width=1pt] (5,-1)-- (5,0);
\begin{scriptsize}
\draw [fill] (3,-1) circle (2pt) node[anchor=east]{$p_1$};
\draw [fill] (3,0) circle (2pt) node[anchor=east]{$p_2$};
\draw [fill] (3,1) circle (2pt) node[anchor=east]{$p_3$};
\draw [fill] (3.5,2) circle (2pt) node[anchor=east]{$p_8$};
\draw [fill] (2.5,2) circle (2pt) node[anchor=east]{$p_4$};
\draw [fill] (2,3) circle (2pt) node[anchor=south]{$p_5$};
\draw [fill] (2.5,3) circle (2pt) node[anchor=south]{$p_6$};
\draw [fill] (3,3) circle (2pt) node[anchor=south]{$p_7$};
\draw [fill] (3.5,3) circle (2pt) node[anchor=south]{$p_9$};
\draw [fill] (5,-1) circle (2pt) node[anchor=west]{$p_{10}$};
\draw [fill] (5,0) circle (2pt) node[anchor=west]{$p_{11}$};
\end{scriptsize}
\end{tikzpicture}
  \caption{Proximity graph of ${\mathcal B}_{\mathcal{F}^{2}}$}
  \label{figex1}
\end{center}
\end{figure}

 The dicritical configuration $\mathcal{B}_{\mathcal F^2}$ of the foliation $\mathcal{F}^2$ (whose proximity graph is depicted in Figure \ref{figex1}) has
$11$ points: $p_1,\ldots,p_{11}$; five of them ($p_5,\,p_6,\,p_7,\,p_9$ and $p_{11}$) are the infinitely near terminal dicritical singularities. The canonical sheaf of $\mathcal{F}^2$ is $\mathcal{K}_{\mathcal{F}^2}=\mathcal{O}_{\mathbb{F}_{2}}(6,6)$. From the proximity relations among the points of ${\mathcal B}_{{\mathcal F}^2}$ and the equalities
\begin{multline*}
[K_{\widetilde{\mathcal F}^2}-K_{Z_{{\mathcal F}^2}}]=6[F^*] +8[M^*] -8[E_1^*] -8[E_2^*]  -5[E_3^*] -4[E_4^*] -2\sum_{i=5}^7[E_i^*]\\
-[E_8^*]-2[E_9^*] -4[E_{10}^*] -5[E_{11}^*],
\end{multline*}
$$[C_{X_0}]=[F^*] -[E_1^*], [C_{X_1}]=[F^*] -[E_{10}^*]-[E_{11}^*] \mbox{ and } 
[C_{Y_0}]=[M^*] -[E_1^*] -[E_2^*] -[E_3^*],$$
\noindent it can be checked that $\Sigma=\{C_{X_0},C_{X_1},C_{Y_0}\}$ is a restricted set of independent algebraic solutions of $\mathcal{F}^2$. 

Considering parameters $\alpha_1,\alpha_2,\alpha_3, \alpha_4$ and $\alpha_5$ associated, respectively, to the infinitely near terminal dicritical singularities $p_5, p_6, p_7, p_9$ and $p_{11}$, and expressing $\alpha_3,\alpha_4$ and $\alpha_5$ in terms of of $\alpha_1$ and $\alpha_2$ (as explained before Proposition \ref{39}), we deduce (keeping the notation of Proposition \ref{39}) that the divisors $T_\alpha$ are of the form
\begin{multline*}
T_{\alpha}=\frac{2}{3}F^*+M^* -E_1^* -E_2^*-\frac{2}{3}E_3^*-\frac{1}{2}E_4^* -\alpha_1E_5^* -\alpha_2E_6^* -\left(\frac{1}{2}-\alpha_1-\alpha_2\right)E_7^* \\
-\frac{1}{6}E_8^*-\frac{1}{6}E_9^* -\frac{1}{2}E_{10}^* -\frac{1}{2}E_{11}^*,
\end{multline*}
\noindent and $\alpha^\Sigma_{\mathcal{F}^2}=(\frac{1}{6},\frac{1}{6})$. Then, $T^2_{\alpha^\Sigma_{\mathcal{F}^2}}=0$ and $K_{Z_{\mathcal{F}^2}}\cdot T_{\alpha^\Sigma_{\mathcal{F}^2}}=-\frac{1}{3}$, which means that we are under Condition (c) of Algorithm \ref{algoritmo2}. Applying this algorithm, $\gamma=6$ in Step (2) and \[
6 T_{\alpha^\Sigma_{\mathcal{F}^2}}=4F^*+6M^*-6E_1^*-6E_2^*-4E_3^*-3E_4^*-\sum_{i=5}^9E_i^*-3E_{10}^*-3E_{11}^*.\]
The algorithm returns a \ipr of $\mathcal{F}^2$ (of genus $0$):\[
\frac{ X_1^4 Y_0^6 + 2 X_0^3 X_1^3 Y_0^5 Y_1 + X_0^6 X_1^2 Y_0^4 Y_1^2}{X_0 X_1^3 Y_0^6 + X_0^7 X_1^3 Y_0^3 Y_1^3 + 3 X_0^{10} X_1^2 Y_0^2 Y_1^4 + 3 X_0^{13} X_1 Y_0 Y_1^5 + X_0^{16} Y_1^6},\]
which provides a rational first integral of \folc, \[
 \frac{f}{g}=\frac{ x^4+2x^3y+x^2y^2}{x^3+x^3y^3+3x^2y^4+3xy^5+y^6}.\]
\end{example}

\begin{example}\label{ex_ip2}
In this new example, \folc denotes the polynomial foliation on $\C^2$ defined by the $1$-form \[
\omega:=(-4 x^5 y - y^6 - 5 x^4 y^6)dx+(x^2 + x^6 + 6 x y^5 + 6 x^5 y^5)dy.\]
 Consider its extended foliation $\mathcal{F}^{\PP^2}$ to $\mathbb{P}^2$. Its canonical sheaf is $\mathcal{K}_{\mathcal{F}^{\PP^2}}=\mathcal{O}_{\mathbb{P}^2}(9)$ and its dicritical configuration, $\mathcal{B}_{\mathcal{F}^{\PP^2}}=\{p_i\}_{i=1}^{36}$, consists of $36$ points, where $p_1, p_7$ and $p_i$, $i\geq 13$, belong to $\mathbb{P}^2$, $p_j$ belongs to the first infinitesimal neighbourhood of $p_{j-1}$ for all $j\in \{2,\ldots,6\}$ and for all $j\in \{8,\ldots,12\}$. The set of infinitely near terminal dicritical singularities (of cardinality $d=26$) is $\{p_6\}\cup \{p_i\}_{i=12}^{36}$.
  
The proximity relations among the points of ${\mathcal B}_{{\mathcal F}^{\PP^2}}$ and the equalities
$$[K_{\widetilde{\mathcal F}^{\PP^2}}-K_{Z_{{\mathcal F}^{\PP^2}}}]=12[L^*] -2\sum_{i=1}^6[E_i^*] -5[E_7^*] -2[E_8^*] -\sum_{i=9}^{11}[E_i^*] -2\sum_{i=12}^{36}[E_i^*],$$
$$[\widetilde{C}_{X}]=[L^*] -\sum_{i=1}^6[E_i^*],\; [\widetilde{C}_{Y}]=[L^*] - [E_1^*] - [E_7^*] -\sum_{i=33}^{36}[E_i^*]$$
$$\mbox{and}\; [\widetilde{C}_{Z}]=[L^*] - \sum_{i=7}^{12} [E_{i}^*]$$
\noindent allow us to deduce that $\Sigma=\{C_{X},C_{Y},C_{Z}\}$ is a restricted set of independent algebraic solutions of $\mathcal{F}^{\PP^2}$ of length $\sigma=3$. Hence, $\ell=d-\sigma-1=22$. Considering parameters $\alpha_1,\ldots,\alpha_{26}$ associated, respectively, with the infinitely near terminal dicritical singularities $p_{14}, p_{15}, \ldots,p_{36}$, $p_6,p_{12},p_{13}$ and expressing $\alpha_{23},\ldots,\alpha_{26}$ in terms of $\alpha_1,\ldots,\alpha_{22}$, one gets that $\alpha^\Sigma_{\mathcal{F}^{\PP^2}}=(\frac{1}{6},\ldots,\frac{1}{6})\in \mathbb{R}^{22}$ and
$$
T_{\alpha^\Sigma_{\mathcal{F}^{\PP^2}}}=L^* -\frac{1}{6}\sum_{i=1}^{36}E_i^*.
$$
Notice that $T_{\alpha^\Sigma_{\mathcal{F}^{\PP^2}}}^2=0$. Running Algorithm \ref{algoritmo3} for ${\mathcal F}^{\PP^2}$, $\Sigma$ and $g=10$, one obtains  that  
$$\frac{ XYZ^4 + Y^6}{XZ^5 + X^5 Z}$$
is a rational first integral of ${\mathcal F}^{\PP^2}$ of genus $10$ (whose algebraic invariant curves are given by the pencil $(\pi_{{\mathcal F}^{\PP^2}})_{*}|6T_{\alpha^\Sigma_{\mathcal{F}^{\PP^2}}}|$). This provides a rational first integral of \folc: $$\frac{ x y + y^6}{x + x^5}.$$
\end{example}
\medskip 

Our last example explains how to use the modification of Algorithm \ref{algoritmo3}, described in Corollary \ref{rr2}, to decide about the existence of a polynomial first integral of genus $g$.

\begin{example}\label{ex_ip3}
Let \folc be a polynomial foliation on $\C^2$ defined by the $1$-form \[
\omega:=(2x+4x^3 y^3)dx+(3y^2+3x^4y^2)dy.\]
Consider the extended foliation $\mathcal{F}^2$ to the Hirzebruch surface $\mathbb{F}_2$ given by the output of Algorithm \ref{algoritmo} with input $(2,\omega)$. Its canonical sheaf is $\mathcal{K}_{\mathcal{F}^2}=\mathcal{O}_{\fd}(3,2)$ and its dicritical configuration, $\mathcal{B}_{\mathcal{F}^2}=\{p_i\}_{i=1}^{22}$, is such that $p_1, p_{11}, p_{14}, p_{17}, p_{20}$ belong to the support of the divisor $C_{X_0}\cup C_{Y_0}$, $\{p_i\}_{i=1}^{10}$, $\{p_i\}_{i=11}^{13}$, $\{p_i\}_{i=14}^{16}$, $\{p_i\}_{i=17}^{19}$ and $\{p_i\}_{i=20}^{22}$ are chains and the unique satellite points are $p_3$ and $p_4$ (which are both proximate to $p_1$). The infinitely near terminal dicritical singularities are $p_{10},\,p_{13},\,p_{16},\,p_{19}$ and $p_{22}$. 

Here $\Sigma'=\{C_{X_0},C_{Y_0}\}$ is a set of $\mathcal{F}^2$-invariant curves and, from the proximity relations among the points of $\mathcal{B}_{\mathcal{F}^2}$ and the equalities
\begin{multline*}
[K_{\widetilde{\mathcal F}^2}-K_{Z_{{\mathcal F}^2}}]=3[F^*] +4[M^*] -4[E_1^*]-2\sum_{i=2}^4 [E_i^*]  -\sum_{i=5}^9 [E_i^*] -2 [E_{10}^*]- [E_{11}^*] -[E_{12}^*]\\
 -2[E_{13}^*] -[E_{14}^*] -[E_{15}^*] -2[E_{16}^*]-[E_{17}^*] -[E_{18}^*] -2[E_{19}^*] -[E_{20}^*] -[E_{21}^*] -2[E_{22}^*],
 \end{multline*}
$$[\widetilde{C}_{X_0}]=[F^*] -[E_1^*]\mbox{ and } \;[\widetilde{C}_{Y_0}]=[F^*] -[E_1^*] -[E_2^*] -[E_{11}^*] -[E_{14}^*] -[E_{17}^*] -[E_{20}^*],$$ 
it can be checked that $\Sigma=\{C_{X_0}\}\subset \Sigma'$ is a restricted set of independent algebraic solutions of $\mathcal{F}^2$ and $\Sigma'$ is not. Considering parameters $\alpha_1,\ldots,\alpha_{5}$ associated, respectively, to the infinitely near terminal dicritical singularities $p_{13},\,p_{16},\,p_{19},\,p_{22}$ and $p_{10}$, and expressing  $\alpha_{5}$ in terms of $\alpha_1,\ldots,\alpha_{4}$ (as explained before Proposition \ref{39}), one gets  that $\alpha^\Sigma_{\mathcal{F}^2}=(\frac{1}{3},\frac{1}{3},\frac{1}{3},\frac{1}{3})$ and
$$
T_{\alpha^\Sigma_{\mathcal{F}^2}}=\frac{2}{3}F^*+M^*-E_1^*-\frac{1}{3}\sum_{i=2}^{22}E_i^*.
$$
Since $T_{\alpha^\Sigma_{\mathcal{F}^2}}^2=0$, running the modification of Algorithm \ref{algoritmo3} given in Remark \ref{rr1}, for $g=5$, one obtains  that  
$$\frac{ X_1^2Y_0^3+X_0^8Y_1^3+X_0^4X_1^4Y_1^3}{X_0^2Y_0^3}$$
is a rational first integral of ${\mathcal F}^2$ of genus $5$ (whose algebraic invariant curves are given by the pencil $(\pi_{{\mathcal F}^2})_{*}|3T_{\alpha^\Sigma_{\mathcal{F}^2}}|$). This provides a polynomial first integral of \folc, which is $x^2+y^3+x^4y^3$.

\end{example}

\bibliographystyle{plain}
\bibliography{MIBIBLIO}

\end{document}